\documentclass[11pt]{amsart}
\usepackage{amsmath, amssymb, amscd, mathrsfs, url, pinlabel,verbatim}
\usepackage[pagebackref]{hyperref}
\usepackage[margin=1.25in,marginparwidth=1in,centering,letterpaper,dvips]{geometry}
\usepackage{color,dcpic,latexsym,graphicx,epstopdf,comment}
\usepackage[all]{xy}
\usepackage[dvipsnames]{xcolor}
\usepackage{tikz,pgfplots}
\usepackage{enumitem,upquote,tabularx,textcomp}
\usepackage[all]{hypcap}
\usepackage[color=blue!20!white,textsize=tiny]{todonotes}

\title{ Floer homology and right-veering monodromy}

\author[John A. Baldwin]{John A. Baldwin}
\address{Department of Mathematics \\ Boston College}
\email{john.baldwin@bc.edu}

\author{Yi Ni}
\address{Department of Mathematics\\California Institute of Technology}
\email{yini@caltech.edu}

\author[Steven Sivek]{Steven Sivek}
\address{Department of Mathematics\\Imperial College London}
\email{s.sivek@imperial.ac.uk}

\thanks{JAB was supported by  NSF FRG Grant DMS-1952707, YN was supported by NSF Grant DMS-181190.}

\makeatletter
\newtheorem*{rep@theorem}{\rep@title}
\newcommand{\newreptheorem}[2]{%
\newenvironment{rep#1}[1]{%
 \def\rep@title{#2 \ref{##1}}%
 \begin{rep@theorem}}%
 {\end{rep@theorem}}}
\makeatother

\newtheorem {theorem}{Theorem}
\newreptheorem{theorem}{Theorem}
\newtheorem {lemma}[theorem]{Lemma}
\newtheorem {proposition}[theorem]{Proposition}
\newtheorem {corollary}[theorem]{Corollary}

\numberwithin{equation}{section}
\numberwithin{theorem}{section}

\theoremstyle{definition}

\newtheorem{remark}[theorem]{Remark}
\newtheorem*{remark*}{Remark}

\newlist{pcases}{enumerate}{1}
\setlist[pcases]{
  label=\bf{Case~\arabic*:}\protect\thiscase.~,
  ref=\arabic*,
  align=left,
  labelsep=0pt,
  leftmargin=0pt,
  labelwidth=0pt,
  parsep=0pt
}
\newcommand{\case}[1][]{%
  \if\relax\detokenize{#1}\relax
    \def\thiscase{}%
  \else
    \def\thiscase{~#1}%
  \fi
  \item
}

\newcommand{\Z}{\mathbb{Z}}

\newcommand{\R}{\mathbb{R}}
\newcommand{\C}{\mathbb{C}}

\newcommand{\F}{\mathbb{F}}
\newcommand{\Q}{\mathbb{Q}}
\newcommand{\spc}{\operatorname{Spin}^c}
\newcommand{\spinc}{\mathfrak{s}}
\newcommand{\spint}{\mathfrak{t}}

\newcommand{\Img}{\operatorname{Im}}

\newcommand{\T}{\mathbb{T}}

\DeclareMathOperator{\Sym}{Sym}

\newcommand\hf{\mathit{HF}}
\newcommand\hfs{\hf^{symp}}
\newcommand\cf{\mathit{CF}}
\newcommand\cfs{\cf^{symp}}
\newcommand\hfk{\mathit{HFK}}
\newcommand\cfk{\mathit{CFK}}
\newcommand\hfkhat{\widehat{\hfk}}
\newcommand\cfkhat{\widehat{\cfk}}

\DeclareFontFamily{U}{mathx}{\hyphenchar\font45}
\DeclareFontShape{U}{mathx}{m}{n}{
      <5> <6> <7> <8> <9> <10>
      <10.95> <12> <14.4> <17.28> <20.74> <24.88>
      mathx10
      }{}
\DeclareSymbolFont{mathx}{U}{mathx}{m}{n}
\DeclareFontSubstitution{U}{mathx}{m}{n}
\DeclareMathAccent{\widecheck}{0}{mathx}{"71}

\newcommand{\hfhat}{\widehat{\mathit{HF}}}
\newcommand{\cfhat}{\widehat{\mathit{CF}}}
\newcommand{\cfkinfty}{\mathit{CFK}^\infty}
\newcommand{\hfp}{\mathit{HF}^+}
\newcommand{\cfp}{\mathit{CF}^+}

\newcommand{\id}{\operatorname{id}}

\newcommand{\cone}{\operatorname{Cone}}

\newcommand{\pt}{\mathrm{pt}}

\newcommand{\mirror}[1]{\overline{#1}}

\usetikzlibrary{calc,intersections}
\tikzset{every picture/.style=thick}
\tikzset{link/.style = { white, double = black, line width = 1.75pt, double distance = 1.25pt, looseness=1.75 }}
\tikzset{crossing/.style = {draw, circle, dotted, minimum size=0.5cm, inner sep=0, outer sep=0}}
\pgfplotsset{compat=1.12}

\begin{document}

\begin{abstract}
We prove that the knot Floer complex of a fibered knot  detects whether the monodromy of its fibration is right-veering. In particular, this leads to a purely  knot Floer-theoretic characterization of tight contact structures, by the work of Honda--Kazez--Mati{\'c}. Our proof makes use of the relationship between the Heegaard Floer homology of mapping tori and the symplectic Floer homology of area-preserving surface diffeomorphisms. We describe applications of this work to Dehn surgeries and taut foliations.
\end{abstract}

\maketitle
\section{Introduction}
\label{sec:intro}

Let $K$ be a fibered knot in a closed 3-manifold $Y$, with fiber $S$ and monodromy  $h:S\to S$. The map $h$ is  said to be \emph{right-veering} if it sends every properly embedded arc in $S$ to the right  (see \S\ref{sec:background} for a precise definition). This dynamical notion is   important in low-dimensional topology  due  to    the following celebrated theorem of Honda--Kazez--Mati{\'c}   \cite{hkm-veering}:

\begin{theorem}\label{thm:hkm}
A contact 3-manifold $(Y,\xi)$ is tight if and only if  every fibered knot $K\subset Y$  supporting $(Y,\xi)$ has right-veering monodromy.\footnote{Technically, Honda--Kazez--Mati{\'c}'s result says that $(Y,\xi)$ is tight iff every fibered \emph{link} supporting $(Y,\xi)$ has right-veering monodromy, but this implies the statement here, as discussed in \S\ref{ssec:nt}.}\end{theorem}

Our goal in this paper is to prove that the knot Floer complex of a fibered knot completely  detects whether its monodromy is right-veering, as described below.\footnote{Among  many other things, this gives a  simple algorithm  to decide whether a diffeomorphism of a surface with connected boundary is right-veering, which can be unwieldy to check from the original definition.}

Recall that a fibered knot $K\subset Y$  gives rise to a filtration of the Heegaard Floer complex of $-Y$. Up to filtered chain homotopy equivalence, this filtration takes the form \[0= \mathscr{F}_{-1-g}\,\subset\,\F\langle\mathbf{c}\rangle = \mathscr{F}_{-g} \,\subset\, \mathscr{F}_{1-g}\, \subset\, \cdots \,\subset \mathscr{F}_{g} = \cfhat(-Y),\] where $g=g(K)$ and we   work with coefficients in $\F = \Z/2\Z$ throughout. In particular, the knot Floer homology groups associated with $K\subset -Y$ are  given by \[\hfkhat(-Y,K,i) \cong H_*(\mathscr{F}_i/\mathscr{F}_{i-1}).\] 
Note that  \[c(\xi) = [\mathbf{c}]\in H_*(\cfhat(-Y))= \hfhat(-Y)\] is the  contact  invariant of the contact manifold $(Y,\xi)$ supported by the fibered knot $K\subset Y$, as defined by Ozsv{\'a}th--Szab{\'o} in  \cite{osz-contact}. If this contact  invariant vanishes (as, for example, when $\xi$ is overtwisted), then the class $[\mathbf{c}]$ vanishes in the homology of some filtration level.
In \cite{baldwin-velavick}, Baldwin--Vela-Vick introduced a numerical invariant $b(K) \in \mathbb{N}\cup \{\infty\}$ to record  the lowest level at which this occurs,
\[b(K) = b(K\subset Y):= \begin{cases}
\infty, & c(\xi) \neq 0\\
g + \min\{k \mid [\mathbf{c}] = 0 \textrm{ in } H_*(\mathscr{F}_k)\}, & c(\xi)=0.
\end{cases}\] 
Moreover, they proved \cite[Theorem 1.8]{baldwin-velavick} that:

\begin{theorem}
\label{thm:bvv}
A fibered knot $K\subset Y$ has right-veering monodromy  if $b(K)>1$. 
\end{theorem}

Beyond  its relevance to contact geometry, this theorem has been critical for knot detection results in both Floer homology \cite{baldwin-velavick} and Khovanov homology \cite{bs-trefoil,bdlls,bhs-cinquefoil}.

Our main result is the  converse of Theorem \ref{thm:bvv}:

\begin{theorem}\label{thm:main} A fibered knot $K\subset Y$ has right-veering monodromy if and only if $b(K)>1$.\end{theorem}

Theorem \ref{thm:bvv} was proved in \cite{baldwin-velavick} by careful inspection of a Heegaard diagram adapted to the fibered knot $K\subset Y$. Its converse (our main  result) is substantially more difficult and  considerably more surprising. In particular, it is unclear how to prove this converse by similarly direct, Heegaard-diagrammatic means (see \S\ref{ssec:hfk} for  discussion about this). Instead, our proof of Theorem \ref{thm:main} is   highly novel, blending the recently-established relationships between knot Floer homology and symplectic Floer homology (as in \cite{bhs-cinquefoil,ni-monodromy, ni-monodromy2}; see also \cite{ghiggini-spano}), with a new criterion proved here (Theorem \ref{thm:symp})  which shows for the first time that  symplectic Floer homology can detect whether a monodromy is right-veering. 

\begin{remark}\label{rmk:alternative}
There are  other useful  formulations of the invariant $b(K)$. 
For example, the cycle $\mathbf{c}\in\mathscr{F}_{-g}$ represents  a class in every page of the spectral sequence  \[E_1 \cong \hfkhat(-Y,K)\implies \hfhat(-Y) \cong E_\infty\] associated with  the filtration above, and $E_{b(K)+1}$ is the first page in which this class vanishes. In particular, $b(K) = 1$ if and only if the spectral sequence differential  \[d_1:\hfkhat(-Y,K,1-g)\to \hfkhat(-Y,K,-g)\] is nonzero. By the symmetry of knot Floer homology under orientation reversal, this holds in turn if and only if the corresponding spectral sequence differential \[d_1:\hfkhat(Y,K,g)\to \hfkhat(Y,K,g-1)\] is nonzero. 
\end{remark}

\subsection{Applications}

One application of Theorem \ref{thm:main}, in combination with Theorem \ref{thm:hkm}, is the following purely knot Floer-theoretic characterization of tightness: 

\begin{corollary}
\label{cor:tighthfk} A contact 3-manifold $(Y,\xi)$ is tight if and only if every fibered knot $K\subset Y$ supporting $(Y,\xi)$ satisfies $b(K)>1$.
\end{corollary}

Another application of our main result is a partial answer to a question posed in  \cite[Question 8.2]{gage-questions} concerning the monodromies of slice fibered knots. The \emph{fractional Dehn twist coefficient} of  a monodromy $h:S\to S$ measures the twisting near $\partial S$ in the free isotopy between $h$ and its  Nielsen--Thurston representative. Inspired by Gabai's notion of \emph{degeneracy slope}, this  coefficient quantifies just how right-veering (or not)   $h$ is, and contains important  information about the associated contact structure \cite{hkm-veering}. 

\begin{remark}
\label{rmk:rvfdtc}For example, if $h$ is neither right-veering nor left-veering, then $h$ has fractional Dehn twist coefficient  equal to zero \cite[Corollary 2.6]{kazez-roberts}.\footnote{A monodromy is \emph{left-veering} if and only if its inverse is right-veering.} \end{remark}

Inspection of low-crossing examples in the knot tables suggests that    monodromies of slice fibered knots  have fractional Dehn twist coefficient zero. However, this is not necessarily the case: as noted in \cite[\S8]{gage-questions}, the $(p,1)$-cable of any   slice fibered knot is  slice and fibered but has fractional Dehn twist coefficient  $1/p$. The authors therefore ask \cite[Question 8.2]{gage-questions} whether the   twist coefficient is always zero for \emph{hyperbolic} slice fibered knots. We do not completely answer this question, but we prove the following closely related corollary, stated in terms of the tau invariant in Heegaard Floer homology (which vanishes for slice knots). 

\begin{corollary}\label{cor:slicefibered}If $K\subset S^3$ is a   fibered knot with thin knot Floer homology satisfying \[|\tau(K)|< g(K),\] then the monodromy of $K$ is neither right-veering nor left-veering.\end{corollary}

\begin{remark}
A fibered knot $K\subset S^3$ satisfies $|\tau(K)|< g(K)$ if and only if neither $K$ nor its mirror is  strongly quasipositive \cite[Theorem 1.2]{hedden-positivity}. 
\end{remark}

Since $|\tau(K)|<g(K)$ is satisfied for every nontrivial slice knot, Corollary \ref{cor:slicefibered} helps explain the  observations about  fractional Dehn twist coefficients of low-crossing slice fibered knots  as many of these have thin knot Floer homology. 

Corollary \ref{cor:slicefibered} also has applications to Dehn surgery. A knot $K\subset S^3$ is called \emph{persistently foliar} if for every $r\in \Q$ there exists a co-orientable taut foliation of the knot complement meeting the boundary transversally in curves  of slope $r$. Note that every nontrivial Dehn surgery on a persistently foliar knot admits a co-orientable taut foliation. It is known that fibered knots whose monodromies are neither right-veering nor left-veering  are persistently foliar \cite{roberts, roberts-composite}.  Therefore, Corollary \ref{cor:slicefibered}  implies the following:

\begin{corollary}\label{cor:slicefiberedtaut}If $K\subset S^3$ is a   fibered knot with thin knot Floer homology satisfying \[|\tau(K)|< g(K),\] then $K$ is persistently foliar.\end{corollary}

\begin{remark}
This corollary is consistent with the L-space conjecture, since fibered knots which are not strongly quasipositive do not admit nontrivial L-space surgeries.
\end{remark}

Fibered alternating knots with $|\tau(K)|< g(K)$ satisfy the hypotheses of Corollary \ref{cor:slicefiberedtaut}. By \cite[Proposition 3.7]{ni-alternating}, these are precisely the fibered alternating knots which are not connected sums of positive torus knots of the form $T_{2,2n+1}$ or  the mirrors of such connected sums. Such knots are thus persistently foliar. In particular, we can use this  to prove that Dehn surgeries on fibered alternating knots satisfy part of the L-space conjecture:

\begin{corollary}
\label{cor:taut} Suppose that $K\subset S^3$ is a fibered alternating knot and $r\in \Q$. Then $S^3_r(K)$ admits a co-orientable taut foliation if and only if it is not an L-space.
\end{corollary}

There is also a   diagrammatic way of proving that  the fibered alternating knots considered above are persistently foliar, to be explained in forthcoming work of Delman--Roberts \cite{dm-alternating}.\footnote{Their results pertain to non-fibered alternating knots as well.} 
We note, however, that the hypotheses of Corollary \ref{cor:slicefiberedtaut} are  satisfied by many non-alternating knots as well.
For instance,  quasialternating knots have thin knot Floer homology  \cite{mo-qa}. Among  the eleven non-alternating prime knots with nine or fewer crossings,  seven   are fibered and quasialternating and satisfy $|\tau(K)|< g(K)$, according to KnotInfo  \cite{knotinfo}: \[8_{20}, 8_{21},9_{43},9_{44},9_{45},9_{47},9_{48}.\]
By Corollary \ref{cor:slicefiberedtaut}, these knots are  therefore persistently foliar.\footnote{Of the 50 non-alternating prime knots with \emph{ten} or fewer crossings, 26 are fibered and quasialternating and satisfy $|\tau(K)|< g(K)$, and are therefore persistently foliar.}

Lastly, by combining Theorem \ref{thm:main} and Corollary \ref{cor:slicefibered} with  work of Ni in \cite[Theorem 1.1]{ni-exceptSurg}, we also obtain the following result about exceptional surgeries.

\begin{corollary}\label{cor:ExceptSurg}
Let   $K\subset S^3$ be a hyperbolic fibered knot such that $S^3_{r}(K)$ is non-hyperbolic for some rational number $r=p/q$ with $\gcd(p,q)=1$.  \begin{itemize}
\item If $\tau(K) = g(K)$, then $0\le r\le4g(K)$.
\item If $\tau(K) = -g(K)$, then $-4g(K)\le r\le 0$.
\item If $|\tau(K)|< g(K)$ and the knot Floer homology of $K$ is thin, then $|q|\le2$.
\end{itemize}
\end{corollary}

We remark that the first two statements only require Theorem \ref{thm:bvv} while the last statement requires the full strength of Theorem \ref{thm:main}.

We provide a detailed sketch of our proof of Theorem \ref{thm:main} below.

\subsection{Proof outline} \label{ssec:proofoutline}Suppose  that $K\subset Y$ is a fibered knot with right-veering monodromy $h:S\to S$.  One can show directly that Theorem \ref{thm:main} holds when $h$ is isotopic to the identity map rel boundary, so let us assume that $h\not\sim \id$. We wish to prove that $b(K)>1$. Let us suppose for a contradiction that $b(K)=1$. 

Let $L\subset S^1\times S^2$ be a fibered knot which represents a contact structure $\xi$ on $S^1\times S^2$ with nontorsion $\spc$ structure $\spinc_\xi$.  Let $L_+\subset S^1\times S^2$ be the $(3,3n+1)$-cable of $L$ for $n$ large, and let  $g_+:F \to F$ denote the monodromy of $L_+$.
Let $L_-\subset S^1\times S^2$ denote the ``mirror" of $L_+$, with monodromy $g_-:F \to F$ given by the inverse of $g_+$.

Since $h\not\sim \id$ is right-veering, it  follows that the monodromy $h^{-1}$ of the mirror $K\subset -Y$ is not right-veering. Therefore, $b(K\subset -Y)=1$ by Theorem \ref{thm:bvv}. In particular,   there is a nontrivial spectral sequence differential \[\hfkhat(Y,K,1-g) \to \hfkhat(Y,K,-g),\] as in Remark \ref{rmk:alternative}. Since $b(K\subset Y)=1$, there is likewise a nontrivial differential \[\hfkhat(Y,K,g) \to \hfkhat(Y,K,g-1).\] We show that the nontriviality of these  differentials implies that the Heegaard Floer  groups of $0$-surgeries on the knots $J_\pm = K\# L_\pm$ in the 3-manifold $Z= Y\#(S^1\times S^2)$ have the same dimensions in their ``next-to-top" $\spc$ summands (see Proposition \ref{prop:yi}), \begin{equation}\label{eqn:hfgroups}\dim\hfp(Z_0(J_+),\textrm{top}{-}1) = \dim\hfp(Z_0(J_-),\textrm{top}{-}1).\end{equation} Our proof   is inspired by  those of \cite[Proposition 3.1]{ni-monodromy} and \cite[Proposition 4.1]{ni-monodromy2}, and uses the  $0$-surgery formula  in Heegaard Floer homology (our requirement that $L$ represents a nontorsion $\spc$ structure, and our taking the cables $L_\pm$ helps when applying this  formula).

Note that the manifold $Z_0(J_\pm)$ is homeomorphic to the mapping torus of the reducible homeomorphism \[h \cup g_\pm: S\cup F \to S\cup F.\] Since $g(S\cup F)\geq 3$, the Heegaard Floer   groups  above are isomorphic to the symplectic Floer homology groups of these homeomorphisms, 
\begin{equation}\label{eqn:hfhfs}\hfp(Z_0(J_\pm),\textrm{top}{-}1) \cong \hfs(h\cup g_\pm),\end{equation} by  work of Lee--Taubes \cite{lee-taubes} and Kutluhan--Lee--Taubes \cite{klt1}. 
These symplectic Floer  groups  can be computed from certain standard form representatives of the mapping classes of $h\cup g_\pm$, as  in \cite{cotton-clay}.  From an analysis of these standard representatives, we prove (see Theorem \ref{thm:symp}) that \[\dim \hfs(h\cup g_+)  = 2+\dim \hfs(h\cup g_-),\] contradicting \eqref{eqn:hfgroups} and \eqref{eqn:hfhfs}. This proves Theorem \ref{thm:main}. 
Implicit in the final step  is a means by which symplectic Floer homology   detects whether a mapping class is right-veering. This is one of the key new insights in this paper, and  may be of independent interest.

\subsection{Organization} In \S\ref{sec:background}, we  review  right-veering  homeomorphisms and their importance in contact geometry,   fractional Dehn twist coefficients, knot Floer homology, the definition of $b$,  and  Cotton-Clay's calculation of symplectic Floer homology.  In \S\ref{sec:0surg}, we prove the equality  \eqref{eqn:hfgroups} following  work of Ni in \cite{ni-monodromy, ni-monodromy2}. We  prove Theorem \ref{thm:main} and its corollaries in \S\ref{sec:proof}.

\subsection{Acknowledgements} We thank Andy Cotton-Clay, Nathan Dunfield, Matt Hedden, Ying Hu, Siddhi Krishna, Tye Lidman, Rachel Roberts, and Shea Vela-Vick for helpful correspondence. We particularly thank Siddhi for pointing out Corollary \ref{cor:slicefiberedtaut}, and for first introducing us to the question posed in \cite[Question 8.2]{gage-questions}.  We are also grateful to the referee for their helpful feedback on the initial version of this paper.

\section{Preliminaries}\label{sec:background} In this section and beyond, \emph{surface}  will refer to a compact, oriented surface with (possibly empty) boundary. All surface homeomorphisms we consider will be orientation-preserving. \emph{Isotopy} of surface homeomorphisms  will refer to  isotopy rel boundary, and will be indicated by $\sim$.
 We will use the term \emph{free isotopy} to refer to  isotopy without boundary constraints.

\subsection{Right-veering homeomorphisms}
\label{ssec:nt}
Suppose $\Sigma$ is a surface with nonempty boundary. Given  two properly-embedded arcs $a,b\subset \Sigma,$ we say that \emph{$a$ is to the right of $b$ at  $p$,} denoted by $a \geq_p b$, if $p$ is  a shared endpoint \[p\in\partial a \cap \partial b \subset \partial \Sigma\] of these arcs, and either:
\begin{itemize}
\item $a$ is  isotopic to $b$ rel boundary, or 
\item after isotoping $a$ rel boundary so that it intersects $b$ minimally, $a$ is to the right of $b$ in a neighborhood of $p$, as shown in Figure \ref{fig:left}. 
\end{itemize}
We say that \emph{$a$ is to the right of $b$}, denoted by $a\geq b$,  if  $a$ is to the right of $b$ at both endpoints. 

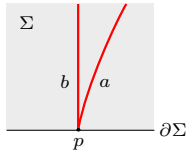
\begin{figure}[ht]
\labellist
\tiny \hair 2pt

\pinlabel $b$ at 32 28
\pinlabel $a$ at 52 27
\pinlabel $p$ at 39 -3
\pinlabel $\Sigma$ at 13 57
\pinlabel $\partial\Sigma$ at 86 4
\endlabellist
\centering
\includegraphics[width=2cm]{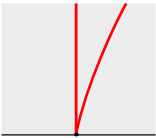}
\caption{$a$ is to the right of $b$  in a neighborhood of  $p$.
}
\label{fig:left}
\end{figure}

Suppose that $\varphi:\Sigma \to \Sigma$ is a homeomorphism of $\Sigma$ which restricts to the identity on a boundary component $B$ of $\Sigma$. Then we say that $\varphi$ is \emph{right-veering at $B$} if \[\varphi(a)\geq_p a\] for every properly-embedded arc $a\subset \Sigma$ and every $p\in\partial a\cap B$.  If $\varphi$ restricts to the identity on all of $\partial \Sigma$, then we say that $\varphi$ is  \emph{right-veering} if  \[\varphi(a) \geq a\]  for every properly-embedded arc in $a\subset \Sigma$; equivalently, if $\varphi$ is right-veering at each boundary component of $\Sigma$. A map  is  \emph{left-veering} if its inverse is right-veering. 

As mentioned in the introduction, the notion of right-veering is important in low-dimensional topology due  to the following theorem of Honda--Kazez--Mati{\'c} \cite{hkm-veering}:

\begin{theorem}\label{thm:hkm2}
A contact 3-manifold $(Y,\xi)$ is tight if and only if  every fibered link $L\subset Y$ supporting $(Y,\xi)$ has right-veering monodromy.\end{theorem}

A version of this result  was stated in Theorem \ref{thm:hkm}  in terms of  fibered \emph{knots} rather than  links; we explain below how Theorem \ref{thm:hkm} follows from Theorem \ref{thm:hkm2}.

\begin{proof}[Proof of Theorem \ref{thm:hkm}] Suppose that $(Y,\xi)$ is tight. By Theorem \ref{thm:hkm2}, every fibered link---in particular, every fibered knot---supporting $(Y,\xi)$ has right-veering monodromy, proving one direction of Theorem \ref{thm:hkm}. 

For the other direction, let us suppose that every fibered knot supporting $(Y,\xi)$ has right-veering monodromy. We must show that $(Y,\xi)$ is tight. By Theorem \ref{thm:hkm2}, it suffices to prove that every fibered link supporting $(Y,\xi)$ has right-veering monodromy. We will prove this by induction on the number of link components  (the base case  holds by assumption). 

Suppose that every fibered link  supporting $(Y,\xi)$ with $n$ components has right-veering monodromy. Let $L\subset Y$ be  a fibered link supporting $(Y,\xi)$ with $n+1$ components, with fiber $\Sigma$ and monodromy $\varphi$. Suppose, for a contradiction, that  $a\subset \Sigma$ is a properly-embedded arc which is not sent to the right by $\varphi$. Let $c\subset \Sigma$ be an arc  disjoint from $\varphi(a)$ whose endpoints lie on two different boundary components of $\Sigma$. Let $\Sigma'$ be the surface obtained from $\Sigma$ by attaching a 1-handle along the endpoints of $c$, and let $\gamma\subset \Sigma'$ be the curve obtained as the union of $c$ with a core of this handle. Letting $D_\gamma$ denote a right-handed Dehn twist about $\gamma$, we observe that the homeomorphism  \[\varphi' = D_{\gamma}\circ \varphi \] does not send $a\subset \Sigma'$ to the right either, and is therefore not right-veering. As the open book $(\Sigma',\varphi')$ is a positive stabilization of $(\Sigma,\varphi)$,   the  associated fibered link $L'\subset Y$ also supports $(Y,\xi)$. But   $L'$ has $n$ components, so  its monodromy $\varphi'$ is  right-veering, a contradiction.
\end{proof}

 \subsection{Fractional Dehn twist coefficients}
 \label{ssec:fdtc}
One can quantify  how right-veering a  homeomorphism is using the notion of \emph{fractional Dehn twist coefficient}, as introduced by Honda--Kazez--Mati{\'c} in \cite{hkm-veering}. We explain this notion below in terms of certain standard representatives of surface homeomorphisms.

Let $\varphi:\Sigma \to \Sigma$ be a homeomorphism of a surface $\Sigma$. We say that $\varphi$ is \emph{periodic} if $\varphi^n = \id$ for some positive integer $n$; when $\chi(\Sigma)<0$,  we will   assume  that $\varphi$ is an isometry with respect to a hyperbolic metric on $\Sigma$ for which $\partial \Sigma$ is a geodesic.
We say instead that $\varphi$ is   \emph{pseudo-Anosov} if there is a transverse pair of singular measured foliations  \[(\mathcal{F}_s,\mu_s)\textrm{ and }(\mathcal{F}_u,\mu_u)\] of $\Sigma$, called the \emph{stable} and \emph{unstable} foliations of $\varphi$, such that \[\varphi(\mathcal{F}_s,\mu_s) = (\mathcal{F}_s,\lambda^{-1}\mu_s) \textrm{ and } \varphi(\mathcal{F}_u,\mu_u) = (\mathcal{F}_u,\lambda\mu_u)\] for some real number $\lambda >1$. These foliations meet $\partial \Sigma$ in a finite number of singular leaves called \emph{prongs}.

Suppose that $\varphi$ is pseudo-Anosov and fixes  a component $B$ of $\partial \Sigma$ setwise. Let $p_1,\dots,p_k$ denote the intersection points of $B$ with the prongs of the stable foliation of $\varphi$, ordered according to the orientation of $B$. Then there is an integer $n$ such that \[\varphi(p_m) = p_{m+n}\] for all $m$, where the subscripts are taken mod $k$. The restriction of $\varphi$ to $B$ is thus isotopic rel $\{p_1,\dots,p_k\}$ to a rotation by $2\pi n/k$. One can perturb $\varphi$ via isotopy rel  $\{p_1,\dots,p_k\}$  in a standard way,   described  in  \cite[\S4.2]{cotton-clay}, to a  smooth map which restricts   to $B$ as a rotation by $2\pi n/k$ on the nose. We will henceforth assume when talking about pseudo-Anosov maps that they are of this perturbed form. In particular, we will assume that both periodic and pseudo-Anosov homeomorphisms of a surface restrict to periodic maps on the boundary.
We use these notions to define standard representatives of surface homeomorphisms below, closely following \cite[Definition 4.6]{cotton-clay}.

Suppose that $\varphi:\Sigma\to\Sigma$ is a homeomorphism of a surface $\Sigma$. By Thurston's classification of surface homeomorphisms \cite{thurston-diffeomorphisms}, $\varphi$ is isotopic to a homeomorphism $\phi$ of   the following form. There is a finite union $N$ of disjoint closed annuli and curves in $\Sigma$ which is invariant under $\phi$ and $\phi^{-1}$ such that: \begin{itemize}
\item \label{def:item1}If $A$ is  an annulus component of $N$, and $\ell$ is the smallest positive integer such that $\phi^\ell(A)=A$, then  $\phi^\ell|_A$ is either a \emph{twist map} or a \emph{flip-twist map}. That is, with respect to an identification \[A \cong [0,1]\times \R/\Z,\] the map $\phi^\ell|_A$ takes one of the following two forms:
\begin{align*}
\textrm{(twist) } \,\,\,\,\,& (q,p) \mapsto (q,p+f(q)),\\
\textrm{(flip-twist) } \,\,\,\,\,& (q,p) \mapsto (1-q,-p-f(q)), \end{align*}
where $f:[0,1]\to \R$ is a strictly monotonic smooth map such that $\phi^\ell|_A$ restricts to a periodic map on every boundary component of $A$ which is disjoint from $\partial \Sigma$. \\

\item Let $A$ and $\ell$ be as above. If $\ell=1$ and $\phi|_A$ is a twist map, then $\textrm{Im}(f)\subset [0,1]$. Such an annulus $A$ is called a \emph{twist region}, and is \emph{positive} or \emph{negative} if $f$ is increasing or decreasing, respectively; the condition on $\mathrm{Im}(f)$ implies that $\phi$ has no fixed points in the interior of $A$. We require that parallel twist regions have the same sign. If $\ell=1$ and $\phi|_A$ is a flip-twist map, then $A$ is called a \emph{flip-twist region}.\\
 
\item Let $S$ be the closure of a component of $\Sigma\setminus N$, and $\ell$  the smallest positive integer such that $\phi^\ell(S)=S$. Then $\phi^\ell|_S$ is either  periodic or pseudo-Anosov. We call $S$ a \emph{fixed component} if $\ell=1$ and  $\phi|_S = \id$. We require that $S$ is fixed if it is an annulus. In particular, parallel twist regions are separated by fixed annuli. A \emph{multitwist region} is a maximal annular subsurface of $\Sigma$ given as a union  of twist regions and the fixed annuli between them.\\
 
\item $N$ is called the \emph{invariant set for $\phi$}. We will assume that it is minimal with respect to inclusion. In particular, there is no curve component of $N$ which abuts a fixed region on both sides. There is a canonical such $N$ up to isotopy.
\end{itemize}
The map $\phi$ is called a \emph{standard representative} of $\varphi$. 

\begin{remark}
A multitwist region $R$  consists of some number $k\geq 0$ of twist regions on which $\phi$ is a full Dehn twist, and at most two twist regions,  each at an end of $R$, on which $\phi$ is a partial Dehn twist. In particular, if $R$ abuts a  boundary component on which $\phi$ is the identity, then $R$ has at most one partial twist region, at an end  interior to the surface, as described below and illustrated in Figure \ref{fig:stdform}. We will encounter  multitwist regions with up to two partial twist regions in the proof of Theorem \ref{thm:symp},  as illustrated in Figure \ref{fig:gluing}.
\end{remark}

\begin{remark}
\label{rmk:twists}
Suppose  $\phi_0$ and $\phi_1$ are standard  form  homeomorphisms of surfaces $\Sigma_0$ and $\Sigma_1$, respectively. Let $B_i$ be a boundary component of $\Sigma_i$ on which $\phi_i$ is the identity and which abuts a  twist region $A_i$ for $\phi_i$, for $i=0,1$. Let \[\Sigma=\Sigma_0\bigcup_{B_0=B_1} \Sigma_1\] be the surface obtained by gluing $\Sigma_0$ to $\Sigma_1$ along these boundary components.  If  $A_0$ and $A_1$ are   twist regions of the same sign, then a standard  representative $\phi$ of the induced map on $\Sigma$ is given by  the union of  the maps $\phi_0$ and $\phi_1$ on either side of a fixed annulus  inserted between $A_0$ and $A_1$. The new fixed annulus is required by the condition  $\mathrm{Im}(f)\subset [0,1]$ in the definition of  a twist region above; in particular, $A_0$ and $A_1$ do not glue to form a single twist region for $\phi$ in this case. By contrast, if $A_0$ and $A_1$ are twist regions of opposite signs, then $\phi$ does not include this additional fixed annulus; the union of $A_0$ and $A_1$ in this case \emph{is} a single twist region for $\phi$. This observation explains  the difference between the standard representatives $\phi_\pm$ in one case in the proof of Theorem \ref{thm:symp}.
\end{remark}

Suppose that $\varphi:\Sigma\to\Sigma$ is a homeomorphism of a surface $\Sigma$ which restricts to the identity on a boundary component $B$, and let $\phi$ be a standard representative of $\varphi$. The \emph{fractional Dehn twist coefficient of $\varphi$ at $B$}, denoted by \[c_B(\varphi)\in\Q,\] is defined as follows. If $B$ does not abut a multitwist region for $\phi$, then \[c_B(\varphi)=0.\] If $B$ does abut a multitwist region $R$, then for some integer $k\geq 0$ and some $\epsilon \in \{\pm 1\}$,  $R$ is a union of $k+1$ twist regions \[A_1,\dots,A_{k+1}\cong [0,1]\times \R/\Z,\] together with the $k$ fixed annuli between them, where:
\begin{itemize}
\item $B$ is identified with $ \{0\}\times \R/\Z\subset A_1,$
\item   $\phi|_{A_i}$ is isotopic  to the map \[(q,p) \to (q, p+\epsilon q),\] for each $i=1,\dots,k$, and 
\item $\phi|_{A_{k+1}}$ is isotopic to the map \[(q,p) \to (q, p+rq)\] for some rational number $r$ with $|r|\in (0,1).$
\end{itemize} That is, $\varphi$ is a full $\epsilon$-twist on each of $A_1,\dots,A_k$, and a partial twist on $A_{k+1}$, as indicated in Figure \ref{fig:stdform}. In this case, we define \[c_B(\varphi) = \epsilon k+r.\] These twist coefficients satisfy \[c_B(\varphi^n) = nc_B(\varphi) \textrm{ and } c_B(\varphi^{-1}) = -c_B(\varphi).\]
If $\Sigma$ has connected boundary, we  denote the twist coefficient simply by $c(\varphi)$. 

\begin{figure}[ht]
\labellist
\tiny \hair 2pt
\pinlabel $R$ at 139 313
\pinlabel $B$ at -1 148
\pinlabel $A_1$ at 70 0
\pinlabel $A_2$ at 135 0
\pinlabel $A_3$ at 200 0
\endlabellist
\centering
\includegraphics[width=2.9cm]{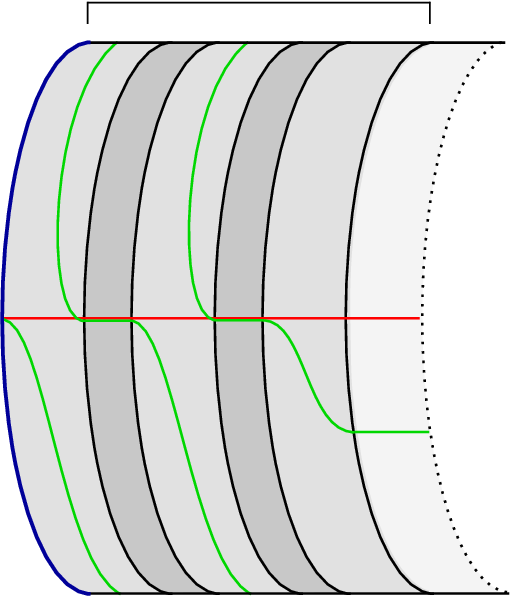}
\caption{A portion of $\Sigma$ near the multitwist region $R$ in the case $k=2$ and $\epsilon = +1$. That is, $R$ is made up of three positive twist regions, $A_1,A_2,A_3$, shaded in medium gray, together with the two fixed annuli between them, shaded in dark gray. The green arc is  the image of the red  under $\phi$. In this example, we see that $c_B(\varphi) \in (2,3)$.
}
\label{fig:stdform}
\end{figure}
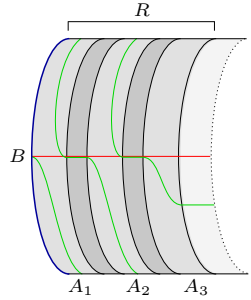

Fractional Dehn twist coefficients are intimately related with the notion of right-veering, as illustrated by Lemma \ref{lem:rvpos} below. This lemma  follows mostly from the work in \cite{hkm-veering}, but appears as stated below (or in a transparently equivalent way) in \cite[Corollary 2.6]{kazez-roberts}.

\begin{lemma}\label{lem:rvpos}
Let $\varphi:\Sigma\to\Sigma$ be a homeomorphism  which restricts to the identity on $\partial \Sigma$, and let $B$ be a boundary component of $ \Sigma$. If $\varphi$ is right-veering at $B$ then $c_B(\varphi)\geq 0$. Moreover, if $c_B(\varphi)> 0$ then $\varphi$ is right-veering at $B$. \qed \end{lemma}

This lemma has the immediate corollary that if $\varphi$ is neither right-veering nor left-veering, then $c_B(\varphi)=0$, as noted in Remark \ref{rmk:rvfdtc}.

\begin{lemma}\label{lem:parv}
Let $\varphi:\Sigma\to\Sigma$ be  a pseudo-Anosov homeomorphism   which restricts to the identity on a boundary component $B$. Then there exists a properly-embedded arc $a\subset \Sigma$ with   $\partial a\subset B$ such that $\varphi(a)\not\geq a$. In particular,  $\varphi$ is not right-veering at $B$. \end{lemma}

\begin{proof}
This follows readily from \cite[Proposition 3.1]{hkm-veering}. In that proposition, however, it is assumed that the map restricts to the identity on each boundary component of the surface. This is not necessarily  the case for  the map $\varphi$  in the lemma, which need not even fix every boundary component of $\Sigma$ setwise. We remedy this by taking an appropriate power. In particular, $\varphi^n$ restricts to the identity on  $\partial \Sigma$ for some positive integer $n$, and  \[c_B(\varphi^n) = nc_B(\varphi)=0.\] Then \cite[Proposition 3.1]{hkm-veering} says that $\varphi^n$ is not right-veering at $B$. Moreover, the proof of that proposition shows that there is a properly-embedded arc $b\subset \Sigma$ with  $\partial b \subset B$ such that $\varphi^n(b) \not \geq b$. It therefore cannot be the case that \[ \varphi^n(b)  \geq \varphi^{n-1}(b) \geq \varphi^{n-2}(b) \geq \dots \geq \varphi(b) \geq b,\] which implies that $\varphi^{i+1}(b) \not\geq \varphi^i(b)$ for some $i$. Letting $a =\varphi^i(b)$, the lemma follows.
\end{proof}

\begin{lemma}\label{lem:periodicid}
If $\varphi:\Sigma\to\Sigma$ is a periodic homeomorphism of a connected surface which restricts to the identity on a boundary component $B$, then $\varphi=\id$.
\end{lemma}

\begin{proof}
We note that $B$ is contained in a connected component of the fixed set of $\varphi$.  Then \cite[Lemma 1.1]{jg} says that such a component must be either (1) a connected component of $\Sigma$, (2) a closed geodesic with a two-sided collar neighborhood, (3) a geodesic arc with endpoints on $\partial\Sigma$, or (4) an isolated point.  The last three cannot contain $B$, so the corresponding component of $\operatorname{Fix}(\varphi)$ must be a component of $\Sigma$, namely $\Sigma$ itself.
\end{proof}

The following lemma is key in our proof of Theorem \ref{thm:symp}, which explains how symplectic Floer homology detects right-veering monodromy.

\begin{lemma}\label{lem:rvstd}
Let $S$ be a surface with one boundary component. Let \[h:S\to S\] be a homeomorphism   which restricts to the identity on the boundary, and let $\alpha$ be a standard  representative of $h$ with invariant set $N$. Let $S_0$ denote the closure of the component of $S\setminus N$ which abuts either $\partial S$ or a multitwist region containing $\partial S$, and let \[\alpha_0 = \alpha|_{S_0}:S_0\to S_0.\] Then $h$ is right-veering if and only if either:
\begin{enumerate}
\item \label{item:pos}$\partial S$ abuts a positive twist region for $\alpha$, or
\item \label{item:id}$\partial S\subset \partial S_0$, $\alpha_0=\id$, and every boundary  component of $S_0$ besides $\partial S$ abuts a positive twist region for $\alpha$.
\end{enumerate}
\end{lemma}

\begin{proof}
Suppose that $h$ is right-veering. Then $c(h) \geq 0$ by Lemma \ref{lem:rvpos}. If $c(h)>0$, then $\partial S$ abuts a positive twist region for $\alpha$, and we are done. Let us therefore suppose  that $c(h)=0$. This implies that $\partial S \subset \partial S_0$, that  $\alpha_0|_{\partial S} = h|_{\partial S}=\id$, and that \[c_{\partial S}(\alpha_0)=c(h)=0.\]
Observe that $h$ is not   freely isotopic to a pseudo-Anosov map, by \cite[Proposition 3.1]{hkm-veering}. If  $h$ is freely isotopic to a periodic map, then Lemma \ref{lem:periodicid} implies that $\alpha_0 = \id$, as desired. There are no boundary components of $S_0$ besides $\partial S$ in this case.

Suppose, therefore, that $h$ is freely isotopic to a reducible map.  If $\alpha_0$ is pseudo-Anosov, then the fact that $c_{\partial S}(\alpha_0)=0$ implies by Lemma \ref{lem:parv}  that there is a properly-embedded arc $a\subset S_0$ with $\partial a \subset \partial S$ such that \[h(a)\sim\alpha_0(a)\not\geq a,\] which contradicts the assumption that $h$ is right-veering. Thus, $\alpha_0$ is periodic, which implies that $\alpha_0 = \id$ by Lemma \ref{lem:periodicid}. It remains to show that  every component $B$ of $\partial S_0\setminus \partial S$ abuts a positive twist region for $\alpha$. It is easy to see  that  the restriction \begin{equation*}\label{eqn:alpha'}\alpha'=\alpha|_{S\setminus \textrm{int}(S_0)}\end{equation*} is  right-veering at $B$, given that $h$ is right-veering and $\alpha_0=\id$. Thus, $c_B(\alpha') \geq 0$ by Lemma \ref{lem:rvpos}. If $c_B(\alpha')>0$, then $B$  abuts a positive twist region. Suppose, for a contradiction,  that $c_B(\alpha')=0$. Then $B$ does not abut any twist region; instead, $B$  abuts a component \[S'\subset S\setminus \textrm{int}(S_0)\] on which $\alpha'$ is either pseudo-Anosov or periodic.  In the first case,  Lemma \ref{lem:parv} says that there is a properly-embedded arc $a\subset S'$ with  $\partial a\subset B$ such that $\alpha'(a)\not \geq a$, but this contradicts the fact that $\alpha'$ is right-veering.  In the second case, Lemma \ref{lem:periodicid} implies that $\alpha'$ restricts to the identity on $S'$. But since $\alpha_0=\id$, this contradicts the minimality of the invariant set for $\alpha$: there should not be a curve abutting two regions on which $\alpha$ is the identity.

For the other direction, suppose first that item (1) of the lemma holds. Then $c(h)>0$, which implies that $h$ is right-veering by Lemma \ref{lem:rvpos}. Suppose now that item (2)  holds. Then the fractional Dehn twist coefficients of the restriction \[\alpha'=\alpha|_{S\setminus \textrm{int}(S_0)}\] are all positive. The map $\alpha'$ is thus right-veering by Lemma \ref{lem:rvpos}. Since $S$ is obtained from $S\setminus \textrm{int}(S_0)$ by attaching 1-handles, and $\alpha$ is the extension of $\alpha'$ to $S$ by the identity,  \cite[Lemma 2.3]{hkm-veering} says that $\alpha$ and therefore $h\sim \alpha$ is right-veering as well.
\end{proof}

\subsection{Symplectic Floer homology} 
\label{ssec:symp} 
Suppose  $\varphi:\Sigma\to\Sigma$ is a homeomorphism of a closed surface $\Sigma$. Let $\omega$ be an area form on $\Sigma$, and let $\phi$ be an area-preserving diffeomorphism of $\Sigma$ isotopic to $\varphi$. Assuming certain nondegeneracy and monotonicity conditions, the symplectic Floer homology $\hfs(\varphi)$ is the homology of a chain complex $\cfs(\phi)$ which is freely-generated as an $\F$-vector space  by the  fixed points of $\phi$, and whose differential counts certain pseudo-holomorphic cylinders. As indicated by the notation and proved by Seidel in  \cite{seidel-hf}, the $\F$-vector space $\hfs(\varphi)$ depends up to isomorphism only on the mapping class of $\varphi$. 

The goal of this section is to review Cotton-Clay's calculation of symplectic Floer homology in terms of standard representatives  (Theorem \ref{thm:cottonclay}). We first establish some  notation.

Let $\varphi:\Sigma\to\Sigma$ be a homeomorphism of a closed surface $\Sigma$, and let $\phi$ be a standard  representative of $\varphi$.  Let $\Sigma_0$ denote the collection of fixed components for $\phi$. Let $\Sigma_1$ be the collection of (non-fixed) periodic components, and let $\Sigma_2$ be the collection of pseudo-Anosov components.
We further divide $\Sigma_0$ into three subcollections as follows. 

Let $\Sigma_a$ be the collection of fixed components for $\phi$ which do not abut any pseudo-Anosov components. Let $\Sigma_{b,p}$ be the collection of fixed components which abut exactly one pseudo-Anosov component, at a boundary with $p$ prongs. Let $\Sigma_{b,p}^\circ$ denote the subsurface obtained from $\Sigma_{b,p}$ by removing an open disk from each component of $\Sigma_{b,p}$. Let $\Sigma_{c,q}$ be the collection of fixed components $S$ which abut at least two pseudo-Anosov components, such that the total number of prongs  meeting the boundary of $S$  is $q$. 

\begin{remark}A fixed component cannot abut a periodic component $S$; otherwise, $\phi$ would restrict to the identity on a boundary component $B$ of  $S$. This would imply by Lemma \ref{lem:periodicid} that $\phi$ is the identity on $S$, violating the minimality of the invariant set for $\phi$.
\end{remark}

We partition $\partial\Sigma_0$ into collections $\partial_\pm\Sigma_0$ of \emph{positive} and \emph{negative} components as follows. Suppose that $S\subset \Sigma_0$ is a fixed component. If a  component of $\partial S$ abuts a positive or negative twist region, then it is assigned to $\partial_\pm \Sigma_0$, respectively. If $S\subset \Sigma_{b,p}$, then the component of $\partial S$ which abuts a pseudo-Anosov component is assigned to $\partial_-\Sigma_0$. If $S\subset \Sigma_{c,q}$, then we assign at least one component of $\partial S$ which meets a pseudo-Anosov component to $\partial_-\Sigma_0$ and at least one other to $\partial_+\Sigma_0$ (and beyond that, it does not matter).

In \cite[\S4.5]{cotton-clay}, Cotton-Clay further perturbs $\phi$ to an area-preserving diffeomorphism  $\hat\phi$ of $\Sigma$ (with respect to some area form) with isolated fixed points, which agrees with $\phi$ on the invariant set. Let $\Lambda(\hat\phi|_{\Sigma_1})$ be the Lefschetz number of $\hat\phi|_{\Sigma_1}$. Let \[\cfs(\hat\phi|_{\Sigma_2})\] denote the symplectic Floer chain complex for $\hat\phi$ restricted to $\Sigma_2$, understood as the $\F$-vector space freely-generated by the fixed points of $\hat\phi|_{\Sigma_2}$ which are not contained in $\partial \Sigma_2$.

Let $n_f$ denote the number of flip-twist regions for $\phi$. 

With this setup, we are finally ready to state Cotton-Clay's formula for symplectic Floer homology \cite[Theorem 4.16]{cotton-clay}, as clarified slightly in \cite[Theorem 1.3]{ni-monodromy2}:

\begin{theorem}\label{thm:cottonclay}
Suppose that $\varphi:\Sigma\to\Sigma$ is a homeomorphism of a closed surface $\Sigma$ with $g(\Sigma)\geq 2$, and let $\phi$ be a standard representative of $\varphi$. Then we have that
\begin{align*}
\hfs(\varphi)\cong\, \,& H_*(\Sigma_a,\partial_-\Sigma_a;\F)\\
&\oplus \bigoplus_p \Big[ H_*(\Sigma_{b,p}^\circ,\partial_-\Sigma_{b,p};\F) \oplus \F^{(p-1)|\pi_0(\Sigma_{b,p})|}\Big]\\
&\oplus \bigoplus_q \Big[ H_*(\Sigma_{c,q},\partial_-\Sigma_{c,q};\F) \oplus \F^{q|\pi_0(\Sigma_{c,q})|}\Big]\\
&\oplus \F^{\Lambda(\hat\phi|_{\Sigma_1})} \oplus \F^{2n_f} \oplus \cfs(\hat\phi|_{\Sigma_2}),
\end{align*}
with respect to the notation introduced above.
\end{theorem}

\begin{remark} \label{rem:relative-homology-surface}
Since the relative homology groups of fixed regions contribute importantly in the formula above, we remind the reader that if $S$ is a connected, oriented surface with boundary, and $\partial_- S$ is subcollection of the components of $\partial S$, then 
\[\dim H_*(S,\partial_-S;\F)=\begin{cases}
 \dim H_*(S;\F),	& \textrm{ if } \partial_-S = \emptyset \textrm{ or } \partial_-S=\partial S,\\
 \dim H_*(S;\F)-2,& \textrm{ otherwise}.
\end{cases}\]
Here ``$\dim H_*$'' refers to the total dimension of homology, rather than the dimension in a particular grading.  We will use this extensively in the proof of Theorem \ref{thm:symp}.
\end{remark}

\begin{remark}\label{rmk:sympinv}
The contributions from $\Sigma_0$, $\Sigma_1$,  $\Sigma_2$ to the formula in Theorem \ref{thm:cottonclay} do not change if we replace $\phi$ with $\phi^{-1}$. In particular, $\hfs(\varphi) \cong \hfs(\varphi^{-1})$ as ungraded $\F$-vector spaces.
\end{remark}

We end by describing the relationship between symplectic Floer homology and Heegaard Floer homology. For this, suppose  that $\varphi:\Sigma\to \Sigma$ is a homeomorphism of a closed surface $\Sigma$  with  $g(\Sigma)\geq 2$. Let $M_\varphi$ denote the mapping torus of $\varphi$. Let us define \[\hfp(M_\varphi,\mathrm{top}{-}1) = \bigoplus_{\substack{\spinc\in\spc(M_\varphi)\\\langle c_1(\spinc),[\Sigma]\rangle = 2g(\Sigma)-4}} \hfp(M_\varphi,\spinc).\] The result below   is a combination of work by Lee--Taubes \cite[Theorem 1.1]{lee-taubes} and Kutluhan--Lee--Taubes \cite[Main Theorem]{klt1}:

\begin{theorem}
\label{thm:hfsymp} Suppose that $\varphi:\Sigma\to\Sigma$ is a homeomorphism of a closed surface $\Sigma$  with $g(\Sigma)\geq 3$. Then \[\hfp(M_\varphi,\mathrm{top}{-}1) \cong \hfs(\varphi).\] 
\end{theorem}

We will apply this theorem in \S\ref{sec:proof} to mapping tori arising as $0$-surgery on the fibered knots $J_\pm\subset Z$ introduced in \S\ref{sec:0surg} in order to prove our main Theorem \ref{thm:main}.

\subsection{Knot Floer homology and $b$}
\label{ssec:hfk} We assume below that the reader has some familiarity with Heegaard Floer homology. Our goals in this section are primarily to establish notation and review the invariant $b$. See \cite{osz-knot, baldwin-velavick} for more background.

Suppose that $(\Sigma,\alpha,\beta,z,w)$ is a doubly-pointed Heegaard diagram for a  nullhomologous knot $K\subset Y$. Recall that the Heegaard  Floer chain complex \[\cfhat(Y)=\cfhat(\Sigma,\alpha,\beta,w)\] is the $\F$-vector space  freely-generated by  intersection points  between the associated tori  \[\T_\alpha,\T_\beta\subset\Sym^k(\Sigma),\] where $k=g(\Sigma)$. The differential  \[\partial:\cfhat(Y)\to \cfhat(Y)\] is the linear map defined on  generators $\mathbf{x}\in \T_\alpha\cap \T_\beta$ by \[\partial(\mathbf{x}) = \sum_{\mathbf{y}\in\T_{\alpha}\cap\T_{\beta}}\sum_{\substack{\phi\in\pi_2(\mathbf{x},\mathbf{y})\\\mu(\phi)=1\\n_w(\phi)=0}} \#\widehat{\mathcal{M}}(\phi)\cdot \mathbf{y},\] where $\pi_2(\mathbf{x},\mathbf{y})$ denotes the set of homotopy classes of Whitney disks from $\mathbf{x}$ to $\mathbf{y}$, $\mu(\phi)$ is  the Maslov index of $\phi$, $n_w(\phi)$ is the intersection number \[\phi\cdot \big(\{w\}\times\Sym^{k-1}(\Sigma)\big),\] and $\widehat{\mathcal{M}}(\phi)$ is the space of pseudo-holomorphic representatives of $\phi$ modulo conformal automorphisms of the domain. The chain homotopy type of this complex, and therefore the (isomorphism type of the) Heegaard Floer homology \[\hfhat(Y)=H_*(\cfhat(Y),\partial),\] is an invariant of $Y$.

Given a Seifert surface $S$ for the knot $K$, each generator $\mathbf{x}$  of the Heegaard Floer complex is assigned an Alexander grading  \[A(\mathbf{x})\in \Z\] such that for generators $\mathbf{x}$ and $\mathbf{y}$ connected by a Whitney disk $\phi\in\pi_2(\mathbf{x},\mathbf{y})$, we have   \begin{equation}\label{eqn:relalex}A(\mathbf{x})-A(\mathbf{y}) = n_z(\phi)-n_w(\phi).\end{equation} 
Let $\mathscr{F}_i$ denote the subspace of $\cfhat(Y)$ spanned by generators in Alexander grading at most $i$. The fact that $n_z(\phi) \geq 0$ when $\phi$ has a pseudo-holomorphic representative, combined with  \eqref{eqn:relalex} and the fact that  $\partial$   counts disks with $n_w(\phi)=0$, implies that 
these subspaces are in fact subcomplexes, and that they define a filtration \[\dots\subset \mathscr{F}_{n-2}\subset\mathscr{F}_{n-1} \subset \mathscr{F}_n=\cfhat(Y).\] The filtered chain homotopy type  of this complex is an invariant of $(Y,K)$ and the relative homology class $[S]\in H_2(Y,K)$. 
We denote by \[\cfkhat(Y,K,[S],i) = \mathscr{F}_i/\mathscr{F}_{i-1}\]  the direct summand of the associated graded complex in Alexander grading $i$, and by \[\hfkhat(Y,K,[S],i) = H_*(\cfkhat(Y,K,[S],i))\] the resulting knot Floer homology group in Alexander grading $i$. Recall that \begin{equation}\label{eqn:vanish}\hfkhat(Y,K,[S],i) = 0 \textrm{ for }|i|>g(S).\end{equation} Letting   \[\hfkhat(Y,K,[S]) = \bigoplus_i\hfkhat(Y,K,[S],i),\]   it follows  that the filtration above gives rise to a  spectral sequence  with \[E_1\cong \hfkhat(Y,K,[S])\textrm{ and } E_\infty \cong  \hfhat(Y)\] whose $d_1$ differential is a sum over integers $i$ of maps of the form \[d_1:\hfkhat(Y,K,[S],i) \to \hfkhat(Y,K,[S],i-1).\] 
The chain complexes  above (and thus the corresponding homology groups) split as  direct sums of complexes over $\spc$ structures on $Y$. Given $\spinc \in\spc(Y)$, we denote by  \[\cfhat(Y,\spinc), \,\hfhat(Y,\spinc),\, \cfkhat(Y,K,[S],\spinc,i),\, \hfkhat(Y,K,[S],\spinc,i)\] the corresponding $\spc$ summands.

\begin{remark}\label{rmk:tau}
For a knot $K\subset S^3$, we have that \[E_\infty \cong  \hfhat(S^3)\cong \F.\] The tau invariant $\tau(K)\in\Z$ is   the  Alexander grading of the generator of this  page.
\end{remark}

\begin{remark}
We will   omit the Seifert surface $S$ from the notation  above where the class $[S]$  is implicit, as in the case of a fibered knot or a knot in a rational homology sphere.
\end{remark}

\begin{remark}\label{rmk:thin}
The knot Floer homology of a knot $K\subset S^3$ is bigraded, \[\hfkhat(S^3,K) = \bigoplus_{m,i} \hfkhat_m(S^3,K,i),\] where $m\in \Z$ denotes the Maslov grading. The knot Floer homology is  \emph{thin} if it is supported in bigradings $(m,i)$ with $m-i$  constant. The fact that the differential $\partial$ shifts Maslov grading by $-1$ implies that the spectral sequence \[E_1\cong \hfkhat(S^3,K)\implies \hfhat(S^3)\cong E_\infty\] collapses at the $E_2$ page when the knot Floer homology of $K$ is thin.
\end{remark}

Suppose now that $K\subset Y$ is a fibered knot of genus $g$.  Then
\begin{equation}\label{eqn:rankone}
\hfkhat(\pm Y,K,g) \cong \hfkhat(\pm Y, K,-g) \cong \F.
\end{equation}
Moreover, if $\spinc$ is the $\spc$ structure on $Y$  associated with  the fibration of $K$ (by which we mean the $\spc$ structure associated with the contact structure on $Y$ supported by $K$), then
\begin{align}\label{eqn:hfkf}\hfkhat(Y,K,\spinc, g)& \cong \F,\\
\label{eqn:hfkzero}\hfkhat(Y,K,\spinc', g)&=0, \textrm{ for } \spinc'\neq \spinc.
\end{align}   Note that the combination of \eqref{eqn:vanish} and \eqref{eqn:rankone}  implies that the  filtration of $\cfhat(-Y)$ associated with  $K\subset -Y$ is filtered chain homotopy equivalent to a filtration of the form \[ 0 = \mathscr{F}_{-1-g}\,\subset\,\F\langle\mathbf{c}\rangle = \mathscr{F}_{-g} \,\subset\, \mathscr{F}_{1-g}\, \subset\, \cdots \,\subset \mathscr{F}_{g} = \cfhat(-Y),\] as mentioned in the introduction. As in that section, we define the invariant \[b(K)=b(K\subset Y)\in\mathbb{N}\cup \{\infty\}\] to be either \[g + \min\{k \mid [\mathbf{c}] = 0 \textrm{ in } H_*(\mathscr{F}_k)\},\]  if $[\mathbf{c}] = 0 \textrm{ in } \hfhat(-Y),$ or $\infty$ otherwise. The spectral sequence interpretation of $b$ in Remark \ref{rmk:alternative} then follows readily from its definition and the discussion above.

Theorem \ref{thm:bvv} is equivalent to the statement that $b(K)=1$ when the monodromy of $K$ is not right-veering. As mentioned in the introduction, this was proved by Baldwin--Vela-Vick via a Heegaard-diagrammatic approach, but  it is not clear how to prove  our main Theorem \ref{thm:main} by a similarly direct strategy. We elaborate on this point below.

The   idea behind Baldwin--Vela-Vick's proof of Theorem \ref{thm:bvv} is roughly the following: suppose that the monodromy $h:S\to S$ of $K$ is not right-veering. Then there is some \emph{basis} arc  $a \subset S$ which is not sent to the right by $h$. This arc and its image $h(a)$ can be used to define attaching curves in a doubly-pointed Heegaard diagram for $K\subset -Y$. The fact that $h$ does not send $a$ to the right is used to find  a generator  \[\mathbf{d}\in\cfhat(-Y),\] in Alexander grading $1-g,$ such that the  sole contribution to $\partial \mathbf{d}$ is a pseudo-holomorphic disk with domain given by a bigon from $\mathbf{d}$ to $\mathbf{c}$. This proves that $b(K)=1$.

One might hope to prove the converse (our Theorem \ref{thm:main}) by similar diagrammatic means. Most naively, given a doubly-pointed Heegaard diagram for $K$ adapted to the open book $(S,h)$ and a basis of arcs on $S$, one might hope that   $b(K) = 1$ implies that there is a bigon from a generator $\mathbf{d}$ as above to $\mathbf{c}$, certifying that at least one of the basis arcs is not sent to the right. However, this naive strategy fails for the reason that one can find a surface $S$, a monodromy $h:S\to S$, and a basis of arcs on $S$ such that $h$ is not right-veering but nevertheless moves every arc in the basis to the right, as illustrated in Figure \ref{fig:nonrvright} below.

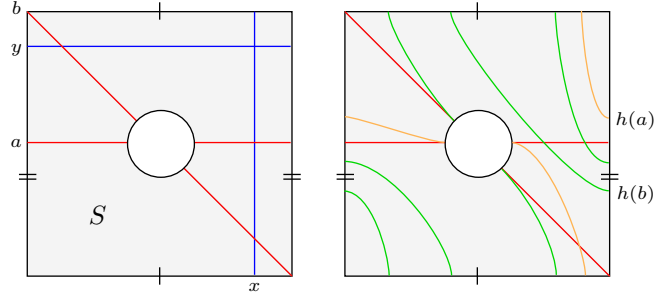
\begin{figure}[ht]
\labellist
\small \hair 2pt
\pinlabel $S$ at 53 48

\tiny
\pinlabel $x$ at 142 7
\pinlabel $y$ at 7 142
\pinlabel $a$ at 7 89
\pinlabel $b$ at 7 166

\pinlabel $h(a)$ at 358 102
\pinlabel $h(b)$ at 358 60

\endlabellist
\centering
\includegraphics[width=8cm]{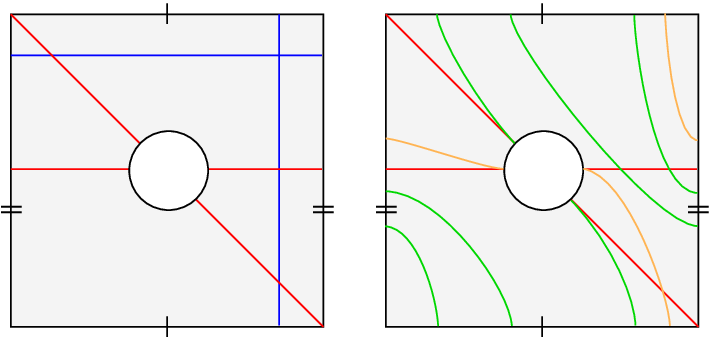}
\caption{The composition  $h=D_x \circ D_{y}^{-1}$ of a right-handed Dehn twist about $x$ with a left-handed Dehn twist about $y$ is a non-right-veering homeomorphism of the once-punctured torus $S$. The arcs $a$ and $b$ form a basis for $S$, and are both moved to the right by $h$.
}
\label{fig:nonrvright}
\end{figure}

What our Theorem \ref{thm:main} ultimately shows of course is that there is \emph{some} basis of arcs for which $b(K)=1$ guarantees the existence of a  bigon as above, but it is not at all clear to us how to find such a basis by  diagrammatic means.

\subsection{The infinity knot Floer complex}\label{ssec:infty}
We end this section with a  review of the $\cfkinfty$ version of the knot Floer complex, which we will use extensively in \S\ref{sec:0surg}.

Given a doubly-pointed Heegaard diagram for a nullhomologous knot $K\subset Y$ with Seifert surface $S$ as in the previous section, the chain complex \[\mathscr{C}=\cfkinfty(Y,K,[S])\] is generated as a vector space over $\F$ by triples  $[\mathbf{x},i,j]$, with $\mathbf{x}\in \T_\alpha\cap \T_\beta$ and $(i,j)\in \Z\oplus \Z$, satisfying \[A(\mathbf{x}) = j-i.\] This complex has the structure of an $\F[U]$-module, where multiplication by $U$ acts as \[U\cdot [\mathbf{x},i,j] =[\mathbf{x},i-1,j-1].\] The differential \[\delta:\mathscr{C}\to\mathscr{C} \] is the $\F[U]$-module map defined on generators by \[\delta([\mathbf{x},i,j]) = \sum_{\mathbf{y}\in\T_{\alpha}\cap\T_{\beta}}\sum_{\substack{\phi\in\pi_2(\mathbf{x},\mathbf{y})\\\mu(\phi)=1}} \#\widehat{\mathcal{M}}(\phi)\cdot [\mathbf{y},i-n_w(\phi),j-n_z(\phi)].\]  The complex $(\mathscr{C},\delta)$ is   $(\Z\oplus\Z)$-filtered  with respect to the  grading which assigns to a generator $[\mathbf{x},i,j]$ the pair $(i,j)$, once again by the nonnegativity of $n_z(\phi)$ and $n_w(\phi)$ for disks $\phi$  which admit pseudo-holomorphic representatives. In particular,  $\delta$ is a sum of maps \[\delta = \sum \delta_{mn}\] over pairs  of nonnegative integers, where $\delta_{mn}$ is the component of $\delta$ which lowers the grading by $(m,n).$
As before, the  filtered chain homotopy type  is an invariant of $(Y,K,[S])$, and this complex  splits as a direct sum of   complexes over $\spc$ structures on $Y$. We will denote by $\mathscr{C}_\spinc$  the  summand  corresponding to  $\spinc\in\spc(Y)$.  

Given $(i,j)\in \Z\oplus\Z$, let  $\mathscr{C}_\spinc(i,j)$  be the  subspace of $\mathscr{C}_\spinc$ spanned by generators of the form $[\mathbf{x},i,j]$.
Then the component $\delta_{mn}$ of $\delta$ restricts to a sum of maps of the form \[\delta_{mn}: \mathscr{C}_\spinc(i,j)\to \mathscr{C}_\spinc(i-m,j-n)\] over pairs of integers $(i,j)$. More generally, given a subset $X\subset \Z\oplus \Z$, we define \[\mathscr{C}_\spinc X = \bigoplus_{(i,j)\in X}\mathscr{C}_\spinc(i,j).\] The differential $\delta$ induces an endomorphism of $\mathscr{C}_\spinc X$ which may or may not be a differential. For example, $\mathscr{C}_\spinc(i,j)$ is naturally a chain complex with respect to the induced map $\delta_{00}$, and there is a canonical isomorphism of this complex with the knot Floer complex  above, \[\mathscr{C}_\spinc(i,j)\cong \cfkhat(Y,K,[S],\spinc,j-i).\] Moreover, for each $k\in \Z$, the induced endomorphism \[\delta_{00}+\delta_{01} + \delta_{02} + \dots\] on $\mathscr{C} \{i=k\}$ is a differential which is filtered by the $j$-coordinate, and this filtered complex is isomorphic to $\cfhat(Y)$ with its filtration induced by $K$ and $[S]$ as  above. The same is true of the complex $\mathscr{C} \{j=k\}$ as filtered by the $i$-coordinate.

In practice, we will use the \emph{reduced model} for $\cfkinfty(Y,K,[S])$. This is the $(\Z\oplus\Z)$-filtered chain complex $(C,d)$ over $\F[U]$, where \[ C = H_*(\mathscr{C},\delta_{00})\] is obtained by taking homology with respect to $\delta_{00}$, and $d$ is the induced differential on $C$. Extending the notational conventions above in the obvious way, we have that \[C_\spinc(i,j) \cong \hfkhat(Y,K,[S],\spinc,j-i)\] for each $\spinc\in\spc(Y)$ and $(i,j)\in \Z\oplus \Z$, and \[d=\sum d_{mn}\] is a sum of maps over pairs $(m,n)$ of nonnegative integers which are not both equal to zero, where each component $d_{mn}$ restricts to a map \[d_{mn}:C_\spinc(i,j) \to C_\spinc(i-m,j-n)\] for every  $(i,j)$. In addition, multiplication by $U$ is a map \[U:C_\spinc(i,j) \to C_\spinc(i-1,j-1).\] This reduced complex $(C,d)$ is $(\Z\oplus\Z)$-filtered chain homotopy equivalent to $(\mathscr{C},\delta)$. In particular, for each $k\in \Z$, the  complex $C\{i=k\}$ with filtration induced by the $j$-coordinate, which as a vector space is given by \[C\{i=k\}\,\cong\,\bigoplus_{j\in \Z}\hfkhat(Y,K,[S],j-k), \] is filtered chain homotopy equivalent to $\cfhat(Y)$ with the filtration induced by $K$ and $[S]$ as above. Moreover, the restriction of  \[d_{01} = (\partial_{01})_*\]  to $C\{i=k\}$ is a sum over integers $j$ of maps of the form \[d_{01}:\hfkhat(Y,K,[S],j-k)\to \hfkhat(Y,K,[S],j-k-1),\] and agrees with the $d_1$ differential of the spectral sequence  \[\hfkhat(Y,K,[S])\implies \hfhat(Y).\]  The same holds for $C\{j=k\}$ and $d_{10}$.

\section{The Heegaard Floer homology of $0$-surgery}\label{sec:0surg}

The goal of this section is to prove Proposition \ref{prop:yi} below. As outlined in the introduction, this result is a  step in the proof of Theorem \ref{thm:main}, which we will complete in the next section. We first establish some notation that will be used in this section and the next.

Let $\spinc_+$ be a nontorsion $\spc$ structure on $S^1\times S^2$. Work of Eliashberg \cite{eliashberg-ot} implies that there is a contact structure $\xi$ with  $\spinc_\xi=\spinc_+$. Let $L\subset S^1\times S^2$ be a fibered knot supporting  $\xi$, with fiber $G$. Let $L_+\subset S^1\times S^2$ be the $(3,3n+1)$-cable of $L$ for $n\geq 1$. This cable is naturally fibered, with fiber given by \[F = T \cup G_1 \cup G_2\cup G_3,\] where $T$ is a genus-$3n$ surface with four boundary components, and  the $G_i$  are copies of $G$. In particular, \[g':=g(L_+) = 3g(L)+3n \geq 3.\]  Since $L_+$ is a positive cable, its fibration also represents the $\spc$ structure $\spinc_+$ by \cite[Corollary 1.12]{bev-cabling}. Let  $g_+:F \to F$ denote the monodromy of $L_+$. Then $g_+$ is reducible: it restricts to $T$ as a periodic map of period $9n+3$, and cyclically permutes the $G_i$. 

Note that  $L_+ \subset -(S^1\times S^2)$ has monodromy $g_-:F \to F$ given by the inverse of $g_+$. Its fibration also represents $\mathfrak{s}_+$. For notational convenience, let $L_-\subset S^1\times S^2$ be the image of this knot under an orientation-reversing homeomorphism of $S^1\times S^2$, and let $\spinc_-$ denote the pullback of $\spinc_+$ under this homeomorphism. We will refer to $L_-$ as the ``mirror" of $L_+$.

\begin{lemma}
\label{lem:L2} The fractional Dehn twist coefficients of $g_\pm$ are given by \[c(g_\pm) =   \pm1/(9n+3).\] 
\end{lemma}

\begin{proof}
It is shown in \cite[Proposition 4.2]{kazez-roberts} that the $(p,q)$-cable of a fibered knot, for $p$ and $q$  relatively prime and $|p|>1$, has  fractional Dehn twist coefficient $1/pq$. This is stated there for cables of fibered knots in $S^3$, but the proof is local and applies to cables of fibered knots in any 3-manifold; see also \cite[Lemma 4.2]{ni-monodromy}. 
\end{proof}

The reason we ultimately consider cables is for the following lemma, which follows from work of Hedden \cite{hedden-cabling} and is a key input for Proposition \ref{prop:yi}.

\begin{lemma}
For $n$ sufficiently large, \[\hfkhat(S^1\times S^2, L_\pm,g'-2) = 0,\] and  \[\hfkhat(S^1\times S^2, L_\pm,g'-1) \cong \F \]   is supported in the $\spc$ structures $\spinc_+$ and $\spinc_-$, respectively.
\label{lem:L}
\end{lemma}

\begin{proof} 
A slight adaptation of the proof of \cite[Lemma 3.6]{hedden-cabling} shows that for $n$ sufficiently large, we have that \[\hfkhat(S^1\times S^2, L_+,\spinc,g'-1) \cong \hfkhat(S^1\times S^2, L,\spinc,g)\] for each $\spinc\in\spc(S^1\times S^2)$, and that \[\hfkhat(S^1\times S^2, L_+,g'-2)=0.\] In particular, there is a doubly-pointed Heegaard diagram for the cable $L_+$ such that  there is an isomorphism of \emph{chain complexes}, \[\cfkhat(S^1\times S^2, L_+,\spinc,g'-1) \cong \cfkhat(S^1\times S^2, L,\spinc,g)\] for each $\spinc$, and for which there are no generators  in Alexander grading $g'-2$. Since \[\hfkhat(S^1\times S^2, L,g)\cong \F\] is supported in the $\spc$ structure $\spinc_+$ by \eqref{eqn:hfkf}-\eqref{eqn:hfkzero}, the result for $L_+$ follows. 
The lemma then follows for $L_-$ from the symmetry  \[\hfkhat(S^1\times S^2,L_+,\spinc,j)\cong \hfkhat(-S^1\times S^2,L_+,\spinc,j),\] which holds for each Alexander grading $j$ and each $\spinc\in\spc(S^1\times S^2)$ \cite[\S3]{osz-knot}.
\end{proof}

\begin{remark}
We will  hereafter assume that  $n$ is  large enough that the conclusion of Lemma \ref{lem:L}  holds.
\end{remark}

The rest of this section is devoted to proving Proposition \ref{prop:yi} below. Our proof is inspired by  the proofs of \cite[Proposition 3.1]{ni-monodromy} and \cite[Proposition 4.1]{ni-monodromy2}.

\begin{proposition}
\label{prop:yi} Suppose $K\subset Y$ is a nontrivial fibered knot of genus $g$ satisfying \[b(K\subset Y) = b(K\subset -Y)=1,\] and let \[J_\pm=K\#L_\pm \textrm{ and }Z = Y\#(S^1\times S^2).\] Then \[\dim\hfp(Z_0(J_\pm),\mathrm{top}{-}1)  =\dim\hfkhat(Y,K,g-1)-1.\] 
\end{proposition}

\begin{proof} 
Let us denote the genus of $J_+$ by \[\bar{g}:=g(J_+) =g(L_+) + g(K) = g'+g.\]
Recall that there is a natural identification \[\spc(Z_0(J_+)) \cong \spc(Z)\times \Z.\] More precisely,  for every $\spc$ structure $\spint$ on $Z$ and each integer $k$, there is a unique $\spc$ structure $\spint_k$ on $Z_0(J_+)$ determined by the conditions that \[\spint_k|_{Z\setminus J_+} = \spint|_{Z\setminus J_+} \textrm{ and } \langle c_1(\spint_k), [S\cup F] \rangle = 2k,\] where we are viewing $S\cup F$ as the result of capping off the boundary connected sum   $S\natural F$, which is the natural fiber surface for the knot $J_+ = K \# L_+$. Moreover, every $\spc$ structure on $Z_0(J_+)$ arises in this way. Recall that  \[\hfp(Z_0(J_+),\mathrm{top}{-}1)\] is the direct sum of the Heegaard Floer  groups of $Z_0(J_+)$ over  $\spc$ structures of the form $\spint_{\bar g-2}$. We will first show that this group is in fact supported in $\spc$ structures $\spint_{\bar g-2}$, where \[\spint\in\spc(Z = Y\# (S^1\times S^2))\] is of the form $\spint = \spinc \# \spinc_+$. We will then prove the dimension formula in the proposition.

Let $(C, d)$ be the reduced model for the $(\Z\oplus \Z)$-filtered knot Floer complex   \[\cfkinfty(Z,J_+),\] as described in \S\ref{ssec:infty}. In particular, for each $\spint \in \spc(Z)$, we have that  \begin{equation}\label{eqn:Chfk}C_\spint(i,j)  \cong \hfkhat(Z,J_+,\spint,j-i),\end{equation} and $d = \sum d_{mn}$, where each component $d_{mn}$ is a sum of maps of the form \[C_\spint(i,j) \to C_\spint(i-m,j-n).\]  This is a complex over $\F[U]$, where multiplication by $U$ is a map \[U:C_\spint(i,j) \to C_\spint(i-1,j-1).\] For each integer $k$, we consider the induced chain complexes \begin{align*}A^+_{k,\spint} &= C_\spint\{\max(i,j-k) \geq 0\},\\
B^+_{\spint}&= C_\spint\{i\geq 0\},
\end{align*}
as in \cite{osz-integer}, where the latter is chain homotopy equivalent to $\cfp(Z,\spint)$. There are two natural chain maps \[v^+_{k}, h^+_k:A^+_{k,\spint}\to B^+_{\spint}\] where $v^+_k$ is  vertical projection onto $C_\spint\{i\geq 0\}$; and $h^+_k$ is horizontal projection onto $C_\spint\{j\geq k\}$, followed by the identification of the latter with $C_\spint\{j\geq 0\}$ induced by multiplication by $U^k$, followed by a chain homotopy equivalence from $C_\spint\{j\geq 0\}$ to $C_\spint\{i\geq 0\}$.

One can compute the  Floer homology of surgeries in terms of this data. For instance, the  Heegaard Floer  complex of $Z_0(J_+)$ in the $\spc$ structure $\spint_k$ is  known to be chain homotopy equivalent to the mapping cone of $v^+_{k}+h^+_k$, \begin{equation}\label{eqn:Zzero2}\cfp(Z_0(J_+),\spint_k) \simeq \cone(v^+_{k}+h^+_k).\end{equation}  For a torsion $\spc$ structure $\spint$, this follows exactly as in \cite[\S4.8]{osz-integer}. For nontorsion $\spint$, this  follows from  \cite[Theorem 3.1]{ni-propg}. We  use this extensively below.

Let us suppose first that  $\spint = \spinc \# \spinc'$ with $\spinc'\neq \spinc_+$. As mentioned above, our aim in this case is to prove  that \begin{equation}\label{eqn:zerotorsion}\hfp(Z_0(J_+),\spint_{\bar g-2}) = 0.\end{equation} This will follow if we can show that \begin{equation}\label{eqn:zerotorsionhat}\hfhat(Z_0(J_+),\spint_{\bar g-2}) = 0,\end{equation} given the  exact triangle \[ \dots \to \hfhat(Z_0(J_+),\spint_{\bar g-2})  \to \hfp(Z_0(J_+),\spint_{\bar g-2})  \xrightarrow{U} \hfp(Z_0(J_+),\spint_{\bar g-2})  \to \dots\] and the fact that every element in the group in \eqref{eqn:zerotorsion} is in the kernel of $U^m$ for some positive integer $m$. 
Let \begin{align*}\widehat{A}_{k,\spint} &= C_\spint\{\max(i,j-k) = 0\},\\
\widehat{B}_{\spint}&= C_\spint\{i= 0\},
\end{align*}
 denote the kernels of $U$ acting on $A^+_{k,\spint}$ and $B^+_{\spint}$, respectively, and let \[\widehat{v}_{k}, \widehat{h}_k:\widehat{A}_{k,\spint}\to \widehat{B}_{\spint}\] be the restrictions of  $v^+_{k}$ and $h^+_k$ to $\widehat{A}_{k,\spint}$. Then we have that \begin{align*}
 \cfhat(Z_0(J_+),\spint_k) &= \ker(U: \cfp(Z_0(J_+),\spint_k)\to \cfp(Z_0(J_+),\spint_k))\\
 &\simeq \ker(U: \cone(v^+_{k}+h^+_k)\to \cone(v^+_{k}+h^+_k))\\
 &=\cone(\widehat{v}_{k}+\widehat{h}_k).
 \end{align*} To prove \eqref{eqn:zerotorsionhat}, it therefore suffices to prove that $\widehat{v}_{\bar g-2}+\widehat{h}_{\bar g-2}$ is an isomorphism.
 
 We first claim that $\widehat{A}_{\bar g-2,\spint}= \widehat{B}_{\spint}$ and hence that the projection $\widehat{v}_{\bar g-2}$ is the identity map. For this, it (more than) suffices  to prove that \begin{equation}\label{eqn:zeroc}C_\spint\{i\leq 0, j= \bar{g}-2\}=C_\spint\{i= 0, j\geq \bar{g}-2\}=0,\end{equation} since in this case we will have by definition that \begin{equation}\label{eqn:Astd}\widehat{A}_{\bar g-2,\spint}= \widehat{B}_{\spint} = C_\spint\{i= 0, j< \bar{g}-2\}.\end{equation} According to \eqref{eqn:Chfk}, each complex in \eqref{eqn:zeroc} is isomorphic as a vector space to a direct sum of knot Floer homology  groups of the form \[\hfkhat(Z,J_+,\spint,k)\] with $k\geq\bar{g}-2$. We claim that these knot Floer homology groups vanish. This follows from  an application of the  K\"unneth formula   \cite[Theorem 7.1]{osz-knot}, which implies  that  \begin{equation}\label{eqn:kunneth} \hfkhat(Z,J_+,\spint,\bar{g}-i) =\bigoplus_{k=0}^i \hfkhat(Y,K,\spinc,g-k)\otimes \hfkhat(S^1\times S^2,L_+,\spinc',g'+k-i)\end{equation} for any integer $i$. The fact that the groups \[\hfkhat(S^1\times S^2,L_+,g')\textrm{ and } \hfkhat(S^1\times S^2,L_+,g'-1)\] are supported in the  $\spc$ structure $\spinc_+\neq \spinc'$, while \[\hfkhat(S^1\times S^2,L_+,g'-2)=0,\] by Lemma \ref{lem:L}, implies that the knot Floer group in \eqref{eqn:kunneth} vanishes for $i=0,1,2$, as claimed.
This proves \eqref{eqn:Astd}, and hence that $\widehat{v}_{\bar g-2}$ is the identity map.
 
Next, since the definition of $h^+_{\bar g-2}$ starts with projection onto $C_\spint\{j\geq \bar g -2\}$, its restriction \[\widehat{h}_{\bar g-2}:\widehat{A}_{\bar g-2,\spint}\to \widehat{B}_{\spint}\] is identically zero, given \eqref{eqn:Astd}. Therefore, \[\widehat{v}_{\bar g-2}+\widehat{h}_{\bar g-2}=\widehat{v}_{\bar g-2}=\id\] is an isomorphism, and hence  $\hfp(Z_0(J_+),\spint_{\bar g-2}) = 0$ for all such $\spint$, as claimed.

Now suppose that  $\spint =\spinc\#\spinc_+$. Since $\spinc_+$ is nontorsion, the evaluation of $c_1(\spint)$ on a sphere factor $\{\pt\}\times S^2$ in the $S^1\times S^2$ summand is nonzero. It follows from the adjunction inequality that
\begin{equation*}\label{eqn:Bzero}H_*(B_\spint^+)\cong \hfp(Z,\spint)=0.\end{equation*} There are two natural  exact triangles coming from  short exact sequences of chain complexes: \[\xymatrix@R=5pt{
\dots\ar[r]^-{}&H_*(A^+_{\bar g-2,\spint})\ar[rr]^-{(v^+_{\bar g-2}+h^+_{\bar g-2})_*} && H_*(B_\spint^+)\ar[r]^-{}&\hfp(Z_0(J_+),\spint_{\bar g-2}) \ar[r]^-{}&\dots\\
\dots\ar[r]^-{}&H_*(A^+_{\bar g-2,\spint})\ar[rr]^-{(v^+_{\bar g-2})_*} && H_*(B_\spint^+)\ar[r]^-{}&H_*(C_\spint\{i<0, j\geq \bar g-2\}) \ar[r]^-{}&\dots}\]
Since $H_*(B_\spint^+)=0$, it follows   that  \begin{equation*}\label{eqn:ZA}\hfp(Z_0(J_+),\spint_{\bar g-2}) \cong  H_*(C_\spint\{i<0, j\geq \bar g-2\})\end{equation*}  for all $\spint\in \spc(Z)$ of the form $\spint = \spinc\#\spinc_+$. 

Note that for any $\spint\in \spc(Z)$,  the complex $C_\spint\{i<0, j\geq \bar g-2\}$ is isomorphic as a vector space to a direct sum of knot Floer  groups of the form \[\hfkhat(Z,J_+,\spint,k)\] with $k\geq \bar{g}-1$. As  above, these groups vanish for  $\spint = \spinc\#\spinc'$ with $\spinc'\neq \spinc_+$. Since we have also shown that \[\hfp(Z_0(J_+),\spint_{\bar g-2})  = 0\] for such $\spint$,  we   conclude that, in fact, \[\hfp(Z_0(J_+),\spint_{\bar g-2}) \cong  H_*(C_\spint\{i<0, j\geq \bar g-2\})\] for \emph{every} $\spint\in \spc(Z)$. Thus, all that remains to prove the formula \begin{equation}\label{eqn:formula}\dim\hfp(Z_0(J_+),\textrm{top}-1) =\dim\hfkhat(Y,K,g-1)-1\end{equation}   in the proposition is to show that \begin{equation}\label{eqn:bigA} \dim H_*(C\{i<0, j\geq \bar g-2\})=\dim \hfkhat(Y,K,g-1)-1. \end{equation} We do so below. 

First note by \eqref{eqn:Chfk}  that the complex $C\{i<0, j\geq \bar g-2\}$ is given by 
 \[\xymatrix@R=13pt{
&C(-1, \bar g-1)\ar[ldd]_{d_{11}} \ar[dd]^-{d_{01}}&&& \F_a\cong \F\ar[ldd]_{d_{11}} \ar[dd]^-{d_{01}}\\
&&\cong&&\\
C(-2,\bar g-2)&\ar[l]^-{d_{10}}C(-1, \bar g-2)&& \F_b\cong \F&\ar[l]^-{d_{10}}\hfkhat(Z,J_+,\bar g-1).} 
\]
The components $d_{01}$ and $d_{10}$ above can be identified with the components of the  $d_1$ differential \begin{align}
\label{eqn:d11}d_1&:\hfkhat(Z,J_+,\bar g) \to \hfkhat(Z,J_+,\bar g-1),\\
\label{eqn:d12}d_1&:\hfkhat(Z,J_+, 1-\bar g) \to \hfkhat(Z,J_+,-\bar g),
\end{align}
respectively, as explained in \S\ref{ssec:infty}. We claim that both components are nontrivial. Indeed, since $b(K\subset Y) = 1$, there is a nontrivial component of the $d_1$ differential \[d_1:\hfkhat(Y,K,g) \to \hfkhat(Y,K,g-1),\] per Remark \ref{rmk:alternative}. The filtered complex associated with the knot $J_+\subset Z$ is filtered chain homotopy equivalent to the tensor product of the filtered complexes associated with $K\subset Y$ and $L_+\subset S^1\times S^2$, by the K\"unneth formula. It follows readily that the differential  in \eqref{eqn:d11} is nontrivial as well. Similarly, we conclude from $b(K\subset -Y)=1$ and Remark \ref{rmk:alternative} that there is a nontrivial  component of the $d_1$ differential \[d_1:\hfkhat(Y,K,1-g) \to \hfkhat(Y,K,-g),\] which shows by the same argument that the differential in  \eqref{eqn:d12} is nontrivial. 

We have thus shown that \[\xymatrix@C=15pt@R=15pt{\\C\{i<0, j\geq \bar g-2\}&\cong&\\}\xymatrix@R=15pt{
&\F_a \ar[dd]^-{d_{01}} \ar[ddl]_-{d_{11}}\\\\
\F_b&\ar[l]^-{d_{10}}\hfkhat(Z,J_+,\bar g-1),}\] where the components $d_{01}$ and $d_{10}$ are injective and surjective, respectively, and $d_{11}$ is either zero or an isomorphism. Letting \[\partial = d_{11} + d_{01} + d_{10},\] we see that the kernel of $\partial$ is the direct sum \[\ker(\partial) = \F_b \oplus \ker(d_{10}),\] while the image of $\partial$ is the direct sum \[\Img(\partial) = \F_b \oplus \Img(d_{01}).\] Since $\ker(d_{10})$ is a codimension-1 subspace of $\hfkhat(Z,J_+,\bar{g}-1)$ and contains $\Img(d_{01})$, we conclude that \[\dim H_*(C\{i<0, j\geq \bar g-2\}) = \dim \hfkhat(Z,J_+,\bar g -1)-2.\] Thus, all that remains for \eqref{eqn:bigA} is to show that \[\dim \hfkhat(Z,J_+,\bar g -1) = \dim \hfkhat(Y,K,g-1)+1.\] By the K\"unneth formula and Lemma \ref{lem:L}, \begin{align*}
\dim \hfkhat(Z,J_+,\bar g -1)&=\dim\hfkhat(Y,K,g-1)\cdot \dim\hfkhat(S^1\times S^2,L_+,g')\\
&\quad+\dim\hfkhat(Y,K,g)\cdot \dim\hfkhat(S^1\times S^2,L_+,g'-1)\\
&=\dim \hfkhat(Y,K,g-1)+1,\end{align*} as desired. This completes the proof of \eqref{eqn:formula} as explained above.
The proof that \[\hfp(Z_0(J_-),\textrm{top}-1) \cong\hfkhat(Y,K,g-1)-1\] proceeds in exactly the same manner.
\end{proof}

\section{Theorem \ref{thm:main} and its corollaries}\label{sec:proof}
In this section, we prove Theorem \ref{thm:main} and its corollaries. Theorem \ref{thm:main} will  follow  from  Proposition \ref{prop:yi} and Theorem \ref{thm:symp} below, as outlined in the introduction. The latter may be viewed as a means by which symplectic Floer homology detects right-veering monodromy. For the statement of the theorem, recall that the maps \[g_\pm:F\to F\] are the monodromies of the fibered knots $L_\pm \subset S^1\times S^2$ introduced in the previous section.

\begin{theorem}
\label{thm:symp}
Let $K\subset Y$ be a fibered knot with monodromy \[h:S\to S\] which is not isotopic to the identity map. Then $h$ is right-veering if and only if  the  maps \[h\cup g_\pm: S\cup F \to S \cup F\] satisfy \[\dim \hfs(h\cup g_+)  =2+ \dim \hfs(h\cup g_-).\]
\end{theorem}

We note that Theorem \ref{thm:main} and its corollaries do not require the \emph{if} direction of Theorem \ref{thm:symp}. We include it here for completeness and because it may be useful for other applications.

\begin{proof}[Proof of Theorem \ref{thm:symp}]
We will apply Theorem \ref{thm:cottonclay} to the homeomorphisms $\varphi_\pm = h\cup g_\pm$ of  the closed surface $\Sigma=S\cup F$. Let $\alpha$ and $\beta_\pm$ be   standard  representatives of $h$ and $g_\pm$, as defined in \S\ref{ssec:fdtc}. Let $\phi_\pm$ be a  standard  representative of $h\cup g_\pm$. 

Note that $\beta_\pm$ are inverses of one another, and thus have the same invariant set $N$.  Since \[c(g_+) = -c(g_-) = 1/(9n+3)\in (0,1),\] $\partial F$ abuts a positive or negative twist region for $\beta_\pm$, respectively. It follows that the components of $F\setminus N$ do not abut in $\Sigma$ the components in the complement of the invariant set for $\alpha$ in $S$, and hence contribute the same  to $\dim \hfs(h\cup g_+)$ as to $\dim \hfs(h\cup g_-)$.

Suppose  that  $h$ is right-veering. We will address the two cases provided by Lemma \ref{lem:rvstd} in turn, beginning with the first: that $\partial S$  abuts a positive twist region for $\alpha$. In this case,  $\phi_+$ has one more fixed annulus than $\phi_-$, as explained in Remark \ref{rmk:twists} and depicted in Figure \ref{fig:gluing}.
Both boundary components of this annulus $A$ abut positive twist regions, so this annulus has no negative boundary components and therefore contributes \[\dim H_*(A;\F) =2\] to the term \[\dim H_*(\Sigma_a,\partial_-\Sigma_a;\F)\] for $\dim\hfs(h\cup g_+)$ in  Theorem \ref{thm:cottonclay}. The remaining contributions to $\dim\hfs(h\cup g_\pm)$ are the same for both, proving the formula in the theorem in this case.

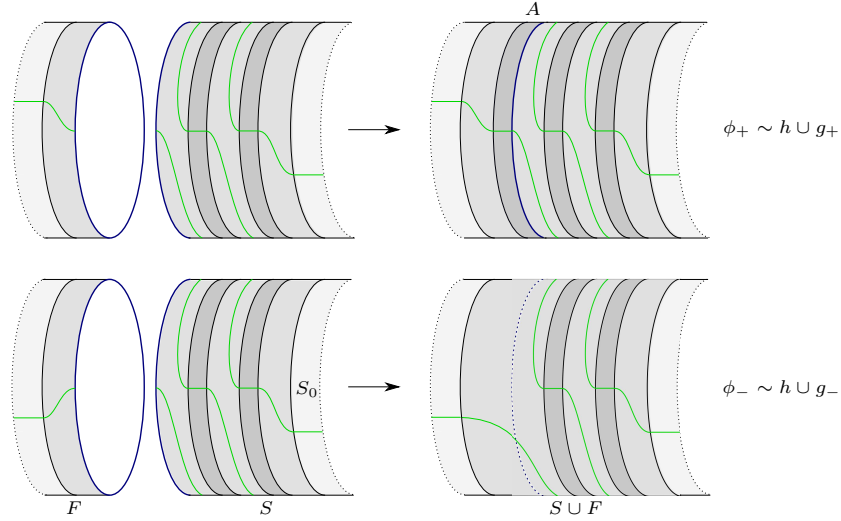
\begin{figure}[ht]
\labellist
\tiny \hair 2pt
\pinlabel $S_0$ at 377 165
\pinlabel $A$ at 654 613

\pinlabel $F$ at 91 -1
\pinlabel $S$ at 327 -1
\pinlabel $S\cup F$ at 707 -1

\pinlabel $\phi_+\sim h\cup g_+$ at 960 465
\pinlabel $\phi_-\sim h\cup g_-$ at 960 145

\endlabellist
\centering
\includegraphics[width=9.3cm]{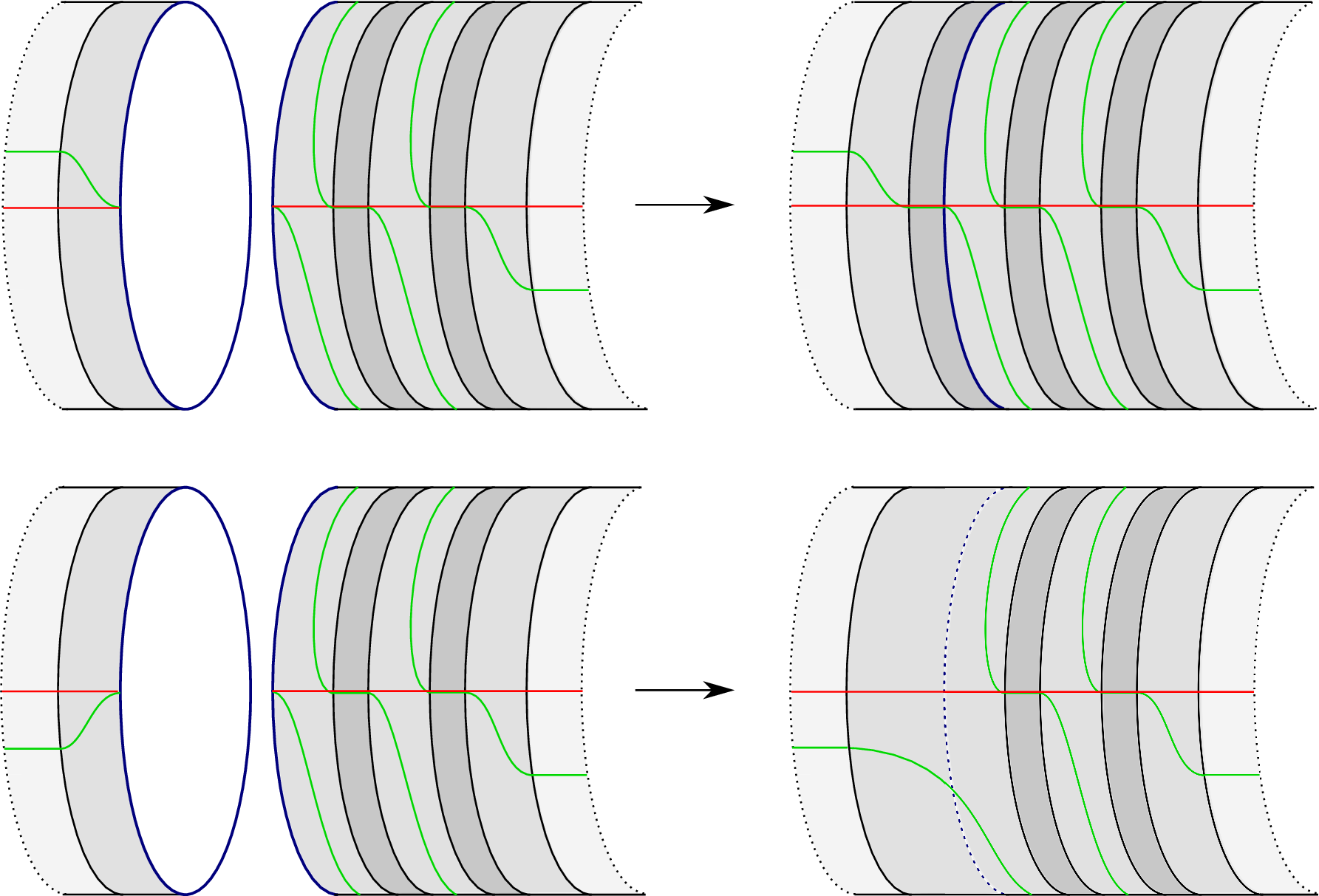}
\caption{The standard representatives $\phi_\pm$ on the top and bottom, respectively. The green arcs are  the images of the red arcs under the corresponding maps. The fixed annuli are shown in dark gray, and the twist regions  in medium gray. Note that $\phi_+$ has one more fixed annulus $A$ than $\phi_-$.}
\label{fig:gluing}
\end{figure}

Suppose next that we are in the second case provided by Lemma \ref{lem:rvstd}:  that $\partial S\subset \partial S_0$, $\alpha_0=\id$, and every boundary  component of $S_0$ besides $\partial S$ abuts a positive twist region for $\alpha$. Since we are assuming for the theorem that $h$ is not isotopic to the identity,  $S_0$ must indeed have boundary components other than $\partial S$. Note that $\partial S$ abuts a positive twist region for $\phi_+$ and a negative twist region for $\phi_-$. Thus, $\partial S_0$ has no negative components for $\phi_+$, but has both positive and negative components for $\phi_-$. It  follows  that the fixed component $S_0$ contributes \[\dim H_*(S_0;\F)\] to  the term \[\dim H_*(\Sigma_a,\partial_-\Sigma_a;\F)\] for $\dim \hfs(h\cup g_+)$ in Theorem \ref{thm:cottonclay} (see Remark~\ref{rem:relative-homology-surface}, with $\partial_-S_0=\emptyset$), but only \[\dim H_*(S_0, \partial S;\F) = \dim H_*(S_0;\F)-2\] to $\dim \hfs(h\cup g_-)$. The remaining contributions to $\dim\hfs(h\cup g_\pm)$ are the same, proving the formula in the theorem in this case as well.

We have so far proven the \emph{only if} direction of the theorem. For the \emph{if} direction (which, as mentioned above, we do not need for our main theorem or its applications), suppose that $h$ is not right-veering. We must show that \[\dim \hfs(h\cup g_+)  \neq 2+ \dim \hfs(h\cup g_-).\] By Lemma \ref{lem:rvstd}, $\partial S$ does not abut a positive twist region for $\alpha$. If $\partial S$ abuts a negative twist region, then the inverse of $h$ is right-veering by Lemma \ref{lem:rvstd}, and we have, by the calculation above and the fact that the dimension of symplectic Floer homology is invariant under taking inverses (see Remark \ref{rmk:sympinv}), that  \[\dim \hfs(h\cup g_+)  = -2+ \dim \hfs(h\cup g_-)\neq 2+ \dim \hfs(h\cup g_-),\] as desired.

If $\partial S$ does not abut a negative twist region, then it does not abut a twist region at all, and we  have that $\partial S \subset \partial S_0$. 
In this case, Lemma \ref{lem:rvstd} says that either $\alpha_0 \neq \id$, or else $\alpha_0 = \id$ and some component of $\partial S_0 \setminus \partial S$ does not abut a positive twist region for $\alpha$. Suppose first that $\alpha_0 \neq \id$. Then, in the notation of Theorem \ref{thm:cottonclay}, $S_0$ belongs to either $\Sigma_1$ (the non-fixed periodic components) or $\Sigma_2$ (the pseudo-Anosov components). In this case, all regions contribute the same amount to both of $\dim \hfs(h\cup g_\pm)$, by Theorem \ref{thm:cottonclay}, and \[\dim \hfs(h\cup g_+)  = \dim \hfs(h\cup g_-) \neq 2+ \dim \hfs(h\cup g_-),\] as desired. Finally, suppose that $\alpha_0 = \id$ and some component $B$ of $\partial S_0 \setminus \partial S$ does not abut a positive twist region for $\alpha$. There are three cases to consider.

 \underline{Case 1: $S_0 \subset \Sigma_a$.} Suppose that  $S_0$ does not abut any pseudo-Anosov components for $\alpha$, so that $S_0$ is a component of $\Sigma_a$ in the notation of Theorem \ref{thm:cottonclay}.  Let \[\alpha'=\alpha|_{S\setminus \textrm{int}(S_0)}.\] We claim that $B$ must abut a negative twist region. Otherwise, $B$ does not abut any twist region, and therefore abuts a periodic component for $\alpha'$. Moreover, $c_B(\alpha') = 0$. This implies by Lemma \ref{lem:periodicid} that $\alpha'$ restricts to the identity  on this component. But that contradicts the minimality of the invariant set for $\alpha$, since $\alpha_0=\id$ as well. Thus, $B$ abuts a negative twist region.  As before, $\partial S$ abuts a positive twist region for $\phi_+$ and a negative twist region for $\phi_-$.  Therefore, $\partial S_0$ has both positive and negative boundary components for $\phi_+$, from which it follows by Remark~\ref{rem:relative-homology-surface} that the fixed component $S_0$ contributes \[\dim H_*(S_0;\F)-2\] to  the term \[\dim H_*(\Sigma_a,\partial_-\Sigma_a;\F)\] for $\dim \hfs(h\cup g_+)$ in Theorem \ref{thm:cottonclay}. Moreover, $S_0$ contributes at least \[\dim H_*(S_0;\F)-2\] to $\dim \hfs(h\cup g_-)$. The remaining contributions to $\dim\hfs(h\cup g_\pm)$ are the same for both, proving that \[\dim \hfs(h\cup g_+)  \leq \dim \hfs(h\cup g_-) < 2+ \dim \hfs(h\cup g_-),\] as desired. 

\underline{Case 2: $S_0 \subset \Sigma_{b,p}$.}
Suppose that $S_0$ abuts exactly one pseudo-Anosov component for $\alpha$, meeting $\partial S_0$ in  $p$ prongs (note that it does not abut a pseudo-Anosov component for  $\beta_\pm$), so that $S_0$ is a component of $\Sigma_{b,p}$ in the notation of Theorem \ref{thm:cottonclay}. Let $\mathring{S_0}$ be the complement of an open disk in $S_0$. Then, by the conventions in \S\ref{ssec:symp}, $\partial_-S_0$ is a nonempty proper subset of $\partial \mathring{S_0}$ for both $\phi_\pm$. It follows  that $S_0$ contributes 
\[\dim H_*(\mathring{S_0};\F)-2\] to the term \[\dim H_*(\Sigma_{b,p}^\circ,\partial_-\Sigma_{b,p};\F)\] in Theorem \ref{thm:cottonclay} for both $\dim \hfs(h\cup g_\pm)$. The remaining contributions to both dimensions are the same, proving that  \[\dim \hfs(h\cup g_+)  = \dim \hfs(h\cup g_-) \neq 2+ \dim \hfs(h\cup g_-),\] as desired. 

\underline{Case 3: $S_0 \subset \Sigma_{c,q}$.}
Suppose that $S_0$ abuts at least two pseudo-Anosov components for $\alpha$, meeting $\partial S_0$ in a total of $q$ prongs, so that $S_0$ is a component of $\Sigma_{c,q}$ in the notation of Theorem \ref{thm:cottonclay}. Then, by the conventions in \S\ref{ssec:symp}, $\partial_-{S_0}$ is a nonempty proper subset of $\partial {S_0}$ for both $\phi_\pm$. It follows  that $S_0$ contributes 
\[\dim H_*({S_0};\F)-2\] to the term \[\dim H_*(\Sigma_{c,q},\partial_-\Sigma_{c,q};\F)\] in  Theorem \ref{thm:cottonclay} for both $\dim \hfs(h\cup g_\pm)$. The remaining contributions to both dimensions are the same, proving that  \[\dim \hfs(h\cup g_+)  = \dim \hfs(h\cup g_-) \neq 2+ \dim \hfs(h\cup g_-),\] as desired.  This completes the proof of the theorem.
\end{proof}

\begin{proof}[Proof of Theorem \ref{thm:main}]
Suppose that $K\subset Y$ is a fibered knot with right-veering monodromy $h:S\to S$. If $h\sim\id$, then $K$ supports the Stein-fillable contact structure on \[Y \cong \#^{2g(S)}(S^1\times S^2),\] which has nontrivial contact invariant. Therefore,  $b(K)>1$, as desired.

Let us now suppose  that $h\not\sim \id$, and let us assume for a contradiction that $b(K\subset Y)=1$. Since $h$ is right-veering and $h\not\sim \id$,  the monodromy $h^{-1}$ of the mirror $K\subset -Y$ is not right-veering. Thus, $b(K\subset -Y)=1$ by Theorem \ref{thm:bvv}. Moreover, $K$ is nontrivial since $h\not\sim\id$.   Proposition \ref{prop:yi} therefore implies that \[\dim\hfp(Z_0(J_+),\mathrm{top}{-}1) = \dim\hfp(Z_0(J_-),\mathrm{top}{-}1),\] where $J_\pm = K\#L_\pm$ and $Z = Y\#(S^1\times S^2)$. Since $Z_0(J_\pm)$ is the mapping torus of $h\cup g_\pm$, and $g(S\cup F)\geq 3$, we have by Theorem \ref{thm:hfsymp} that \[\hfp(Z_0(J_\pm),\textrm{top}{-}1) \cong \hfs(h\cup \phi_\pm).\] Therefore, \[ \dim\hfs(h\cup \phi_+) =\dim \hfs(h\cup \phi_-).\] But this contradicts the conclusion of Theorem \ref{thm:symp}. 
\end{proof}

\begin{proof}[Proof of Corollary \ref{cor:tighthfk}]
This follows immediately from Theorems \ref{thm:hkm} and \ref{thm:main}.
\end{proof}

\begin{proof}[Proof of Corollary \ref{cor:slicefibered}]
As noted in Remark \ref{rmk:tau}, $\tau(K)$ is equal to the Alexander grading of the generator of the $E_\infty\cong \F$ page of the spectral sequence \[E_1\cong \hfkhat(S^3,K)\implies \hfhat(S^3) \cong E_\infty.\] The thinness hypothesis  implies that this spectral sequence collapses at the $E_2$ page, as noted in Remark \ref{rmk:thin}. Thus, every element in the knot Floer homology of $K$ in Alexander grading different from $\tau(K)$ is either 1) a boundary or 2) not a cycle with respect to the $d_1$ differential in the spectral sequence. In particular, since $g:=g(K)\neq \tau(K)$, there
is a nontrivial component of $d_1$ from \[\hfkhat(S^3,K,g)\to \hfkhat(S^3,K,g-1)\] By  Remark \ref{rmk:alternative}, this implies that $b(K\subset S^3) = 1$. The same reasoning applied to the mirror shows that $b(K\subset -S^3)=1$ as well. By Theorem \ref{thm:main}, the monodromies $h^{\pm 1}$  of  $K\subset \pm S^3$ are thus non-right-veering. In particular, $h$ is neither right-veering nor left-veering.
\end{proof}

\begin{proof}[Proof of Corollary \ref{cor:slicefiberedtaut}]
The inequality $|\tau(K)|<g(K)$ means that $K\subset S^3$ is nontrivial, and Corollary \ref{cor:slicefibered} says that the monodromy of $K$ is neither right-veering nor left-veering. Then $K$ is persistently foliar by \cite[Theorem 1.4]{roberts-composite}; we note that the cited theorem is really a slight generalization of \cite[Theorem 4.7 (1)]{roberts}, which is stated using different terminology and only for pseudo-Anosov monodromy. \end{proof}

\begin{proof}[Proof of Corollary \ref{cor:taut}] Suppose $K\subset S^3$ is a fibered alternating knot. First, let us suppose that $K$ is a connected sum of torus knots of the form \[K = T_{2,2n_1+1}\, \#\, \dots \,\#\,T_{2,2n_k+1}.\] 
If $k=1$ and $r$ is a rational number other than $2(2n_1+1),$ then $S^3_r(K)$ is a Seifert manifold with  base  $S^2$. Such manifolds admit co-orientable taut foliations if and only if they are non-L-spaces by \cite[Theorem 1.1]{lisca-stipsicz}. For $k=1$ and $r=2(2n_1+1)$,  $S^3_r(K)$ is a connected sum of lens spaces and therefore an L-space, and it does not admit a taut foliation since it is reducible. If $k>1$, then $K$ does not admit an L-space surgery by \cite[Theorem 1.2]{krcatovich-prime}, and   is persistently foliar by \cite[Theorem 6.1]{roberts-composite}. So, in this case, $S^3_r(K)$ is a non-L-space and admits a co-oriented taut foliation for every $r\in \Q$.

If $K$ is not a connected sum of torus knots, then neither $K$ nor its mirror is strongly quasipositive by \cite[Proposition 3.7]{ni-alternating}. Thus, $K$ does not admit an L-space surgery, and  \[|\tau(K)|<g(K)\] \cite{hedden-positivity}. The latter  implies by Corollary \ref{cor:slicefiberedtaut} that $K$ is persistently foliar. So, in this case, $S^3_r(K)$ is a non-L-space and admits a co-oriented taut foliation for every $r\in \Q$. \end{proof}

\begin{proof}[Proof of Corollary \ref{cor:ExceptSurg}]
Suppose  first that $\tau(K) = g(K)$. From the interpretation of $\tau(K)$ in Remark \ref{rmk:tau}  as the Alexander grading of the  $E_\infty\cong \F$ page of the spectral sequence \[E_1\cong \hfkhat(S^3,K)\implies \hfhat(S^3) \cong E_\infty,\]  we conclude that the generator of $\hfkhat(S^3,K,g)$ must survive in this spectral sequence. In particular, the component of the $d_1$ spectral sequence differential from \[\hfkhat(S^3,K,g)\to \hfkhat(S^3,K,g-1)\] vanishes. Per Remark \ref{rmk:alternative}, this implies that $b(K) >1$, which  implies  by Theorem \ref{thm:main} that the monodromy of $K$ is right-veering.\footnote{For an argument using a bigger hammer, note that $\tau(K)=g(K)$ implies that $K$ is strongly quasipositive and thus supports the tight contact structure on $S^3$, by \cite[Theorem 1.2]{hedden-positivity}. This implies by Theorem \ref{thm:hkm} that the monodromy of $K$ is right-veering.} Then \cite[Theorem 1.1]{ni-exceptSurg} says that $0\le r\le4g(K)$.

Suppose next that $\tau(K) = -g(K)$. The fact that $S^3_r(K)$ is non-hyperbolic implies that \[S^3_{-r}(\mirror{K}) \cong -S^3_r(K)\] is also non-hyperbolic. Since \[\tau(\mirror{K}) = -\tau(K) = g(K)=g(\mirror{K}),\] we have by the previous case that $0\le -r\le4g(K)$, which implies that $-4g(K)\leq r \leq 0$.

Finally, if $|\tau(K)|<g(K)$ and the knot Floer homology of $K$ is thin, then Corollary \ref{cor:slicefibered} tells us that the monodromy of $K$ is neither right-veering nor left-veering. Then \cite[Theorem 1.1]{ni-exceptSurg} says that $|q|\leq 2$.
\end{proof}

\bibliographystyle{alpha}
\bibliography{References}

\newcommand{\etalchar}[1]{$^{#1}$}
\begin{thebibliography}{BEVHM12}

\bibitem[BDL{\etalchar{+}}21]{bdlls}
John~A. Baldwin, Nathan Dowlin, Adam~Simon Levine, Tye Lidman, and Radmila
  Sazdanovic.
\newblock Khovanov homology detects the figure-eight knot.
\newblock {\em Bull. Lond. Math. Soc.}, 53(3):871--876, 2021.

\bibitem[BEVHM12]{bev-cabling}
Kenneth~L. Baker, John~B. Etnyre, and Jeremy Van Horn-Morris.
\newblock Cabling, contact structures and mapping class monoids.
\newblock {\em J. Differential Geom.}, 90(1):1--80, 2012.

\bibitem[BHS21]{bhs-cinquefoil}
John~A. Baldwin, Ying Hu, and Steven Sivek.
\newblock Khovanov homology and the cinquefoil.
\newblock arXiv:2105.12102, 2021.

\bibitem[BS22]{bs-trefoil}
John~A. Baldwin and Steven Sivek.
\newblock Khovanov homology detects the trefoils.
\newblock {\em Duke Math. J.}, 171(4):885--956, 2022.

\bibitem[BVV18]{baldwin-velavick}
John Baldwin and David~Shea Vela-Vick.
\newblock A note on the knot {F}loer homology of fibered knots.
\newblock {\em Algebr. Geom. Topol.}, 18(6):3669--3690, 2018.

\bibitem[CC09]{cotton-clay}
Andrew Cotton-Clay.
\newblock Symplectic {F}loer homology of area-preserving surface
  diffeomorphisms.
\newblock {\em Geom. Topol.}, 13(5):2619--2674, 2009.

\bibitem[CL]{knotinfo}
J.~C. Cha and C.~Livingston.
\newblock Knot{I}nfo: {T}able of {K}not {I}nvariants.
\newblock \url{http://www.indiana.edu/~knotinfo}.
\newblock August 30, 2019.

\bibitem[DR]{dm-alternating}
Charles Delman and Rachel Roberts.
\newblock Nontorus alternating knots are persistently foliar.
\newblock Forthcoming.

\bibitem[DR21]{roberts-composite}
Charles Delman and Rachel Roberts.
\newblock Persistently foliar composite knots.
\newblock {\em Algebr. Geom. Topol.}, 21(6):2761--2798, 2021.

\bibitem[Eli89]{eliashberg-ot}
Y.~Eliashberg.
\newblock Classification of overtwisted contact structures on {$3$}-manifolds.
\newblock {\em Invent. Math.}, 98(3):623--637, 1989.

\bibitem[GS22]{ghiggini-spano}
Paolo Ghiggini and Gilberto Spano.
\newblock {Knot Floer homology of fibred knots and Floer homology of surface
  diffeomorphisms}.
\newblock arXiv:2201.12411, 2022.

\bibitem[Hed05]{hedden-cabling}
Matthew Hedden.
\newblock On knot {F}loer homology and cabling.
\newblock {\em Algebr. Geom. Topol.}, 5:1197--1222, 2005.

\bibitem[Hed10]{hedden-positivity}
Matthew Hedden.
\newblock Notions of positivity and the {O}zsv\'{a}th-{S}zab\'{o} concordance
  invariant.
\newblock {\em J. Knot Theory Ramifications}, 19(5):617--629, 2010.

\bibitem[HKK{\etalchar{+}}21]{gage-questions}
Diana Hubbard, Keiko Kawamuro, Feride~Ceren Kose, Gage Martin, Olga
  Plamenevskaya, Katherine Raoux, Linh Truong, and Hannah Turner.
\newblock Braids, fibered knots, and concordance questions.
\newblock In {\em Research directions in symplectic and contact geometry and
  topology}, volume~27 of {\em Assoc. Women Math. Ser.}, pages 293--324.
  Springer, Cham, [2021] \copyright 2021.

\bibitem[HKM07]{hkm-veering}
Ko~Honda, William~H. Kazez, and Gordana Mati\'{c}.
\newblock Right-veering diffeomorphisms of compact surfaces with boundary.
\newblock {\em Invent. Math.}, 169(2):427--449, 2007.

\bibitem[JG93]{jg}
Bo~Ju Jiang and Jian~Han Guo.
\newblock Fixed points of surface diffeomorphisms.
\newblock {\em Pacific J. Math.}, 160(1):67--89, 1993.

\bibitem[KLT20]{klt1}
\c{C}a\u{g}atay Kutluhan, Yi-Jen Lee, and Clifford~Henry Taubes.
\newblock {$\rm HF{=}HM$}, {I}: {H}eegaard {F}loer homology and
  {S}eiberg-{W}itten {F}loer homology.
\newblock {\em Geom. Topol.}, 24(6):2829--2854, 2020.

\bibitem[KR13]{kazez-roberts}
William~H. Kazez and Rachel Roberts.
\newblock Fractional {D}ehn twists in knot theory and contact topology.
\newblock {\em Algebr. Geom. Topol.}, 13(6):3603--3637, 2013.

\bibitem[Krc15]{krcatovich-prime}
David Krcatovich.
\newblock The reduced knot {F}loer complex.
\newblock {\em Topology Appl.}, 194:171--201, 2015.

\bibitem[LS07]{lisca-stipsicz}
Paolo Lisca and Andr\'{a}s~I. Stipsicz.
\newblock Ozsv\'{a}th-{S}zab\'{o} invariants and tight contact 3-manifolds.
  {III}.
\newblock {\em J. Symplectic Geom.}, 5(4):357--384, 2007.

\bibitem[LT12]{lee-taubes}
Yi-Jen Lee and Clifford~Henry Taubes.
\newblock Periodic {F}loer homology and {S}eiberg-{W}itten-{F}loer cohomology.
\newblock {\em J. Symplectic Geom.}, 10(1):81--164, 2012.

\bibitem[MO08]{mo-qa}
Ciprian Manolescu and Peter Ozsv\'{a}th.
\newblock On the {K}hovanov and knot {F}loer homologies of quasi-alternating
  links.
\newblock In {\em Proceedings of {G}\"{o}kova {G}eometry-{T}opology
  {C}onference 2007}, pages 60--81. G\"{o}kova Geometry/Topology Conference
  (GGT), G\"{o}kova, 2008.

\bibitem[Ni20]{ni-exceptSurg}
Yi~Ni.
\newblock {Exceptional surgeries on hyperbolic fibered knots}.
\newblock arXiv:2007.11774, 2020.

\bibitem[Ni21]{ni-alternating}
Yi~Ni.
\newblock A characterization of {$T_{2g+1,2}$} among alternating knots.
\newblock {\em Acta Math. Sin. (Engl. Ser.)}, 37(12):1841--1846, 2021.

\bibitem[Ni22]{ni-monodromy2}
Yi~Ni.
\newblock {Knot Floer homology and fixed points}.
\newblock arXiv:2201.10546, 2022.

\bibitem[Ni23]{ni-monodromy}
Yi~Ni.
\newblock A note on knot {F}loer homology and fixed points of monodromy.
\newblock {\em Peking Math. J.}, 6(2):635--643, 2023.

\bibitem[Ni24]{ni-propg}
Yi~Ni.
\newblock Property {G} and the 4-genus.
\newblock {\em Trans. Amer. Math. Soc. Ser. B}, 11:120--143, 2024.

\bibitem[OS04]{osz-knot}
Peter Ozsv\'{a}th and Zolt\'{a}n Szab\'{o}.
\newblock Holomorphic disks and knot invariants.
\newblock {\em Adv. Math.}, 186(1):58--116, 2004.

\bibitem[OS05]{osz-contact}
Peter Ozsv\'{a}th and Zolt\'{a}n Szab\'{o}.
\newblock Heegaard {F}loer homology and contact structures.
\newblock {\em Duke Math. J.}, 129(1):39--61, 2005.

\bibitem[OS08]{osz-integer}
Peter~S. Ozsv\'{a}th and Zolt\'{a}n Szab\'{o}.
\newblock Knot {F}loer homology and integer surgeries.
\newblock {\em Algebr. Geom. Topol.}, 8(1):101--153, 2008.

\bibitem[Rob01]{roberts}
Rachel Roberts.
\newblock Taut foliations in punctured surface bundles. {II}.
\newblock {\em Proc. London Math. Soc. (3)}, 83(2):443--471, 2001.

\bibitem[Sei02]{seidel-hf}
Paul Seidel.
\newblock Symplectic {F}loer homology and the mapping class group.
\newblock {\em Pacific J. Math.}, 206(1):219--229, 2002.

\bibitem[Thu88]{thurston-diffeomorphisms}
William~P. Thurston.
\newblock On the geometry and dynamics of diffeomorphisms of surfaces.
\newblock {\em Bull. Amer. Math. Soc. (N.S.)}, 19(2):417--431, 1988.

\end{thebibliography}

\end{document}